%% file: hilbdeg.tex
\documentclass[a4paper, 12pt]{amsart}

\usepackage[latin1]{inputenc}
\usepackage{lmodern}
\usepackage{mathrsfs}
\usepackage[english]{babel}
\usepackage{graphicx}
\usepackage{color}
\usepackage{tikz-cd}
\usepackage{enumerate} 
\usepackage{microtype}

\theoremstyle{plain}
\newtheorem{theorem}{Theorem}[section]
\newtheorem{lemma}[theorem]{Lemma}
\newtheorem{proposition}[theorem]{Proposition}

\theoremstyle{definition}
\newtheorem{definition}[theorem]{Definition}
\newtheorem{remark}[theorem]{Remark}
\newtheorem{example}[theorem]{Example}
\newtheorem*{inputdata}{Input}
\newtheorem*{outputdata}{Output}

\newcommand{\ZZ}{\mathbb{Z}} 
\newcommand{\iso}{\cong}     
\newcommand{\VV}{\mathbb{V}} 
\newcommand{\PP}{\mathbb{P}} 
\renewcommand{\AA}{\mathbb{A}} 
\newcommand{\sheaf}[1]{\mathscr{#1}} 
\newcommand{\OO}{\sheaf{O}}  
\newcommand{\tprod}{{\textstyle\prod}}  
\newcommand{\res}[2]{\left.#1\right|_{#2}} 
\newcommand{\Gm}{\mathbb{G}_m} 
\newcommand{\union}{\cup}
\newcommand{\Union}{\bigcup}
\newcommand{\intsct}{\cap}

\newcommand{\tensor}{\otimes}
\newcommand{\Tensor}{\bigotimes}

\DeclareMathOperator{\pr}{pr}
\DeclareMathOperator{\Pic}{Pic}
\DeclareMathOperator{\SL}{SL}

\DeclareMathOperator{\Spec}{Spec}
\DeclareMathOperator{\Proj}{Proj}
\DeclareMathOperator{\Hilb}{Hilb}
\newcommand{\into}{\hookrightarrow}

\title[Degenerations of Hilbert schemes]{A GIT construction of degenerations of Hilbert schemes of points}

\begin{document}

\author{Martin G. Gulbrandsen}
\address{University of Stavanger\\
Department of Mathematics and Natural Sciences\\
4036 Stavanger\\
Norway}
\email{martin.gulbrandsen@uis.no}

\author{Lars H. Halle}
\address{University of Copenhagen\\ 
Department of Mathematical Sciences\\
Universitetsparken 5\\ 
2100 Copenhagen\\ 
Denmark}
\email{larshhal@math.ku.dk}

\author{Klaus Hulek}
\address{Leibniz Universit\"at Hannover\\
Institut f\"ur algebraische Geometrie\\
Welfengarten 1\\
30060 Hannover\\
Germany}
\email{hulek@math.uni-hannover.de}

\begin{abstract}
We present a Geometric Invariant Theory (GIT) construction which allows us to
construct good projective degenerations of Hilbert schemes of points for simple
degenerations. A comparison with the construction of Li and Wu shows that our
GIT stack and the stack they construct are isomorphic, as are the associated
coarse moduli schemes. Our construction is sufficiently explicit to obtain good
control over the geometry of the singular fibres. We illustrate this by giving a
concrete description of degenerations of  degree $n$ Hilbert schemes of a simple
degeneration with two components. 
\end{abstract}

\maketitle


Constructing and understanding  degenerations of moduli spaces is a crucial
problem in algebraic geometry, as well as a vitally important technique, going
back to the classical German and  Italian schools where it was used for solving
enumerative problems. New techniques for studying degenerations were introduced
by Li and Li--Wu respectively. Their approach is based on the technique of {\em
expanded degenerations}, which first appeared in \cite{li-2001}. This method is
very general and can be used to study degenerations of various types of moduli
problems, including Hilbert schemes and moduli spaces of sheaves. In
\cite{LW-2011} Li and Wu used degenerations of Quot-schemes and coherent
systems to obtain degeneration formulae for Donaldson--Thomas invariants  and
Pandharipande--Thomas stable pairs. The reader can find a good introduction to
these techniques in Li's article \cite{li-2013}.

The motivation for our work was a concrete geometric question: we wanted to
understand degenerations of  irreducible holomorphic symplectic manifolds.
Clearly, a starting point for this is to study de\-gen\-erations of $K3$
surfaces and their Hilbert schemes. Our guiding example were type II
degenerations of $K3$ surfaces, but we  were soon led to investigate the
degeneration of {\em Hilbert schemes of points} for simple degenerations $X \to
C$ where  we make no a priori restriction on the type of the fibre nor its
dimension. A simple degeneration means in particular that the total space is
smooth and that the central fibre $X_0$ over the point $0 \in C$ of the
$1$-dimensional base $C$ has normal crossing along smooth varieties.  The aim
of this paper is to develop the technique for the {\em construction} of
degenerations of Hilbert schemes which give us not only abstract existence
results, but also allow us to control the geometry of the degenerate fibres. In
the  forthcoming paper \cite{GHHZ-2018} we will then investigate the {\em
properties} of these degenerations.  

At this point we would like to explain the common ground, but also the
differences of our approach to that of Li and  Wu. First of all we only
consider Hilbert schemes of points, whereas Li and Wu consider more generally
Hilbert schemes of ideal sheaves with arbitrary Hilbert polynomial, and even
Quot schemes. We have not investigated in how far our techniques can be
extended to non-constant Hilbert polynomials. This might indeed be a question
well worth pursuing, but one which would go far beyond the scope of this paper.
The common ground with the approach of Li and Wu is that we also use Li's
method of expanded degenerations $X[n] \to C[n]$. In the case of constant
Hilbert polynomial the relevance of this construction is the following:
ideally, one wants to construct a family whose special fibre over $0$
parametrizes length $n$ subschemes of the degenerate fibre $X_0$.  Clearly, the
difficult question is how to describe subschemes whose support meets the
singular locus of $X_0$. The main idea of the construction of expanded
degenerations $X[n] \to C[n]$ is that, whenever a subscheme approaches a
singularity in $X_0$, a new ruled component is inserted into $X_0$ and thus it
will be sufficient to work with subschemes supported on the smooth loci of the
fibres of  $X[n] \to C[n]$. The price one pays for this is that the dimension
of the base $C[n]$ is increased at each step of increasing $n$, and finally one
has to take equivalence classes of subschemes supported on the fibres of  $X[n]
\to C[n]$. Indeed, the construction of expanded degenerations also includes the
action of an $n$-dimensional torus $G[n]$ which acts on $X[n] \to C[n]$ such
that $C[n]/\!/{G[n]}=C$.

The way Li and Wu then proceed is by constructing the {\em stack}
$\mathfrak{X}/\mathfrak{C}$ {\em of expanded degenerations} associated to $X\to
C$, which is done by introducing a suitable notion of equivalence on  expanded
degenerations. For fixed Hilbert polynomial $P$ they then introduce the notion
of {\em stable} ideal sheaves with Hilbert polynomial $P$, and use this to
define a stack $\mathcal{I}^P_{\mathfrak{X}/\mathfrak{C}}$ over $C$
parametrizing such stable ideal sheaves. In the case of constant Hilbert
polynomial $P=n$ this leads to subschemes of length $n$ supported on the smooth
locus of a fibre of an expanded degeneration, and having finite automorphism
group. We call the  stack $\mathcal{I}^n_{\mathfrak{X}/\mathfrak{C}}$ the {\em
Li--Wu stack}. For details see \cite{LW-2011} and, for a survey, also
\cite{li-2013}. In contrast to this approach our method does not use the Li--Wu
stack, but is based on a Geometric Invariant Theory approach (GIT,
\cite{GIT}), which we will now outline.

\subsubsection*{Acknowledgements}

The first author would like to thank the Research Council of Norway (NFR) for
partial support under grant 230986. The third author is grateful to Deutsche
Forschungsgemeinschaft (DFG) for partial support under grant Hu 337/7-1.

\subsection{The main results}\label{intro:mainresults}

The main technical achievement of the paper is to construct a suitable set-up
which allows us to apply GIT methods. To perform this we must make one
assumption on the dual graph $\Gamma(X_0)$ associated to the singular fibre
$X_0$, namely that it is \emph{bipartite}, or equivalently it has no cycles of
odd length. This is not a crucial restriction as we can always perform a
quadratic base change to get into this situation. We first construct a
relatively ample line bundle $\sheaf L$ on $X[n] \to C[n]$. The bipartite
assumption allows us to construct a particular $G[n]$-linearization on $ \sheaf
L$ which will then turn out to be well adapted for our applications to Hilbert
schemes. The definition of the correct $G[n]$-linearization is the most
important technical tool of this paper. Using $\sheaf L$ we can construct an
ample line bundle ${\sheaf M}_{\ell}$  on the relative Hilbert scheme
$\mathbf{H}^n := \Hilb^n(X[n]/C[n])$, which comes equipped with a natural
$G[n]$-linearization. (The integer $\ell \gg 0$ only plays an auxiliary role.)
This construction is sufficiently explicit to allow us to analyse GIT
stability, using a relative version of the Hilbert--Mumford numerical criterion
(see \cite[Cor.~1.1]{GHH-2015}). In particular, we prove that (semi-)stability
of a point $[Z] \subset \mathbf{H}^n$ only depends on the degree $n$ cycle
associated to $Z$ (and not on its scheme structure). 

After having fixed the $G[n]$-linearized sheaf $\sheaf L$, our construction
depends a priori on several choices. One choice is the orientation of the dual
graph $\Gamma(X_0)$. As we work with a bipartite graph, it admits exactly two
bipartite orientations and we will show that these lead to isomorphic GIT
quotients. We moreover need to select a suitable $\ell$ in the construction of
$\sheaf M_{\ell}$. Our characterization of stable $n$ cycles will a posteriori
show  that the final result is independent also of this choice. 

This characterization is indeed crucial and in order to formulate this theorem,
we first need some notation. Let $[Z] \in \mathbf{H}^n$ be represented by a
subscheme $Z \subset X[n]_q$ for some point $q \in C[n]$. Using a local \'etale
coordinate $t$ on $C$ we obtain coordinates $t_i, i \in \{1, \ldots n+1 \}$ on
$C[n]$ and we define $\{a_1, \ldots , a_r\} $ to be the subset indexing
coordinates with $t_i(q)=0$. Setting $a_0=1$ and $a_{r+1}=n+1$ we obtain a
vector ${\mathbf a}=(a_0, \ldots, a_{r+1}) \in \ZZ^{r+2}$, which, in turn,
determines a vector $\mathbf{v}_{\mathbf{a}} \in \ZZ^{r+1}$ whose $i$-th
component is $ a_i - a_{i-1} $. 

We say that $Z$ has {\em smooth support} if $Z$ is supported in the smooth part
of the fibre $X[n]_q$. Then each point $P_i$ in the support of $Z$ is contained
in a unique component of  $X[n]_q$ with some multiplicity $n_i$. This allows us
to  define the {\em numerical support}  ${\mathbf v}(Z) \in \ZZ^{r+1}$, see
Definition \ref{def:numericalsupport}. Our characterization then reads as
follows:
\begin{theorem}\label{the:intro-stablelocus}
	Let $\ell \gg 2n^2$. The (semi-)stable locus in $ \mathbf{H}^n $ with respect
	to $\sheaf{M}_{\ell}$ can be described as follows:
	\begin{enumerate}
		\item
			If $[Z] \in \mathbf{H}^n$ has smooth support, then $[Z] \in
			\mathbf{H}^n(\sheaf{M}_{\ell})^{ss}$ if and only if
			\begin{equation*}
				\mathbf{v}(Z) = \mathbf{v}_{\mathbf{a}}.
			\end{equation*}
			In this case, it also holds that $[Z] \in
			\mathbf{H}^n(\sheaf{M}_{\ell})^{s} $.
		\item
			If $[Z] \in \mathbf{H}^n$ does not have smooth support, then $[Z]
			\notin \mathbf{H}^n(\sheaf{M}_{\ell})^{ss}$.
	\end{enumerate}
\end{theorem}
We denote the locus of stable points by
$\mathbf{H}^n_{\mathrm{GIT}}:=\mathbf{H}^n(\sheaf{M}_{\ell})^{s}$ (it does not
depend on $\ell$). It is interesting to note that our GIT approach
independently also leads to the property that stable cycles have smooth
support, a condition also appearing in Li--Wu stability. In fact, GIT stable
cycles are always Li--Wu stable, but  the converse does not hold in general. In
other words we obtain an inclusion $\mathbf{H}^n_{\mathrm{GIT}} \subset
\mathbf{H}^n_{\mathrm{LW}}$ of GIT stable cycles in Li--Wu stable cycles,
which, in general, is strict, see Lemma \ref{lemma-stabilitycomparison} and the
comment following it. 
 
We can now form the {\em GIT-quotient}
\begin{equation*}
	I^n_{X/C} = \mathbf{H}^n_{\mathrm{GIT}}/G[n].
\end{equation*}
This is the main new object which we construct in this paper. The advantage of
our method is that we can control the GIT stable points very explicitly and
this allows us to analyse the geometry of the fibres of the degenerate Hilbert
schemes in great detail. Moreover, we can also use the results of
\cite{GHH-2015}, where it was shown, in particular, that  $I^n_{X/C}$ is
projective over $C$.

We can also form the {\em stack quotient} 
\begin{equation*}
	{\mathcal I}^n_{X/C} = [\mathbf{H}^n_{\mathrm{GIT}}/G[n]].
\end{equation*}
Our main result about this stack is  
\begin{theorem}\label{intro:GITquotient}
	The GIT quotient $ I^n_{X/C} $ is projective over $C$. The stack
	$\mathcal{I}^n_{X/C} $ is a Deligne-Mumford stack, proper and of finite type
	over $C$, having $ I^n_{X/C} $ as coarse moduli space.
\end{theorem}

We also investigate how the GIT stack quotient and the Li--Wu stack compare.
For this we construct a natural morphism $f \colon \mathcal{I}^n_{X/C} \to
\mathcal{I}^n_{\mathfrak{X}/\mathfrak{C}}$ between the two stacks and show
\begin{theorem}\label{intro:comparison}
	The morphism $f \colon \mathcal{I}^n_{X/C} \to
	\mathcal{I}^n_{\mathfrak{X}/\mathfrak{C}} $ is an isomorphism of
	Deligne-Mumford stacks.
\end{theorem}

In this way our approach gives an alternative proof of the  properness over the
base curve $C$ of the Li--Wu stack $\mathcal{I}^n_{\mathfrak{X}/\mathfrak{C}}$
for Hilbert schemes of points, see \cite[Thm.~3.54]{li-2013}. It thus turns out
that our GIT approach and the Li--Wu construction of degenerations of Hilbert
schemes of points are in fact equivalent. The main advantage which we have thus
gained is, in addition to constructing a relatively projective coarse moduli
space for the Li--Wu stack, that we have the tools to explicitly describe the
degenerate Hilbert schemes. We will illustrate this with the example of degree
$n$ Hilbert schemes on two components, which we treat in detail in Section
\ref{sec:examples}.

Of course, one of the main objectives of this research is to construct good
degenerations of, in particular, Hilbert schemes of $K3$ surfaces, such as in
the work of  Nagai \cite{nagai-2008} who used an ad hoc approach which works
very well in degree $2$. At this point it is also worth noting that the simple
approach of taking the relative Hilbert scheme will not work as this is hard to
control and badly behaved. We will study the properties of our degenerations in
detail in \cite{GHHZ-2018}, but we would like to mention at least the main
results in support of our approach. First of all, starting with a strict simple
degeneration $X \to C$ of surfaces, the GIT stack ${\mathcal I}^n_{X/C}$ is
smooth and semi-stable as a $DM$-stack over $C$. The scheme $I^n_{X/C}$ has
finite quotient singularities -- see also Section \ref{sec:examples} for the
degree $2$ case -- and $(I^n_{X/C},(I^n_{X/C})_0)$ is simple normal crossing up
to finite group actions.  In particular, this allows us to attach a dual
complex to the central fibre and due to our good control of the degenerations
we can describe this complex explicitly. If $X \to C$ is a type II degeneration
of $K3$ surfaces, then the stack ${\mathcal I}^n_{X/C}$ carries a nowhere
degenerate relative logarithmic $2$-form. If $X \to C$ is any strict simple
degeneration of surfaces, then $(I^n_{X/C},(I^n_{X/C})_0)$  is a dlt
(divisorial log terminal) pair. Moreover, if we start with a type II Kulikov
model of $K3$ surfaces $I^n_{X/C} \to C$ is a minimal dlt model. In this case
the dual complex can be identified with the Kontsevich--Soibelman skeleton.

Lastly, let us remark that for a simple degeneration $f \colon X \to C$, it is
also natural to consider configurations of $n$ points in the fibres of $f$
(rather than length $n$ subschemes, as in this article). This has been studied
thoroughly by Abramovich and Fantechi in \cite{AF14}. In particular, they
exhibit a moduli space, which is projective over $C$, parametrizing
\emph{stable} configurations.

\subsection{Organization of the paper}\label{intro:organization}

The paper is organized as follows. Section \ref{sec:expandeddeg} introduces
most of the main concepts and technical tools. In particular, we will review
the notions of a simple degeneration $X \to C$ and of expanded degenerations
$X[n] \to C[n]$,  as well as the action of the rank $n$ torus $G[n]$ on $X[n]
\to C[n]$. The construction of $X[n] \to C[n]$  depends on the choice of an
orientation of the dual graph $\Gamma(X_0)$ of the central fibre $X_0$. In
Proposition \ref{prop:localeq} we shall give a concrete description and local
equations for $X[n] \to C[n]$, see also  \cite[\S 4.2]{wu-2007}. We then enter
into a discussion of the properties of $X[n] \to C[n]$. We will prove in
Proposition \ref{prop:scheme-criterion}  that the algebraic space $X[n]$ is a
scheme if and only if the degeneration $X \to C$ is {\em strict}, i.e. all
components of $X_0$ are smooth. Our next aim is to understand when the morphism
$X[n] \to C[n]$ is projective. It turns out that this is the case if and only
if the directed graph $\Gamma(X_0)$ contains {\em no directed cycles}, see
Proposition \ref{prop:projectivity-criterion}. Since we aim at a GIT approach
we need a relatively ample line bundle $\sheaf L$ on $X[n] \to C[n]$ together
with a suitable $G[n]$-linearization. This will be achieved in Section
\ref{sec:linearization} and this construction is the technical core of our
approach. At this point we must impose another condition on the degeneration $X
\to C$, namely that the dual graph $\Gamma(X_0)$ can be equipped with a {\em
bipartite orientation}, see Section \ref{sec:linearization}. In Proposition
\ref{prop:inversion} we shall prove that reversing the orientation of the graph
$\Gamma(X_0)$ leads to an isomorphic quotient. Finally, we investigate the
fibres of the morphism $X[n] \to C[n]$ in detail and enumerate their components
in Proposition \ref{prop:components}, an essential tool for all practical
computations, in particular also for the GIT analysis. 

In Section \ref{sec:GIT-analysis} we perform a careful analysis of GIT
stability. Using the line bundle $\sheaf L$ we construct the relatively ample
line bundle ${\sheaf M}_{\ell}$  on $\mathbf{H}^n = \Hilb^n(X[n]/C[n])$, which
inherits a $G[n]$-linearization. The main result of this section is Theorem
\ref{theorem-stablelocus} (Theorem \ref{the:intro-stablelocus}) where we
characterize the stable locus. For these calculations we will make extensive
use of the local coordinates which we introduce in Section
\ref{sec:localcoordinates}.
 
Section \ref{sec:quotients} is devoted to the comparison of our construction
with the Li--Wu stack. For this we introduce the GIT quotient stack ${\mathcal
I}^n_{X/C}$ and prove the properness Theorem \ref{prop-GITquotient} (Theorem
\ref{intro:GITquotient}).  Finally we construct a map between the GIT quotient
stack and the Li--Wu stack and prove their equivalence in Theorem
\ref{prop:comparison} (Theorem \ref{intro:comparison}). 

In Section \ref{sec:examples} we  discuss one example in detail in order to
illustrate how the machinery works. The example we have chosen is a simple
degeneration $X \to C$ where $X_0$ has two components. We shall describe the
geometry of the central fibre $(I_{X/C})_0$ in detail and, in case of degree
$2$ give a complete classification of the singularities of the total space. We
also compute the dual complex for arbitrary degree $n$, which turns out to be
the  standard $n$-simplex.

\subsection{Notation}

We work over a field $k$ which is algebraically closed of characteristic zero.
By a \emph{point} of a $k$-scheme of finite type, we will always mean a closed
point, unless further specification is given. The projectivization
$\PP(\sheaf{E})$ of a coherent sheaf $\sheaf{E}$ is Grothendieck's
contravariant one, parametrizing rank one quotients.

For an integer $n$ we denote $[n] = \{1, \ldots, n\}$. 


\section{Expanded degenerations}\label{sec:expandeddeg}

Here we recall a construction, due to Li \cite{li-2001}, which to a simple
degeneration $X\to C$ over a curve (Definition \ref{def:simple}), together with
an orientation of the dual graph $\Gamma$ of the degenerate fibre $X_0$,
associates a family of {\em expanded degenerations} $X[n]\to C[n]$ over an
$(n+1)$-dimensional base $C[n]$, equipped with an action by an algebraic torus
$G[n] \iso \Gm^n$.

The aim of the current section is to set the stage for applying GIT to the
induced $G[n]$-action on the relative Hilbert scheme $\Hilb^n(X[n]/C[n])$. Thus,
after recalling Li's construction in Section \ref{sec:expanded-summary} we study
under which circumstances $X[n]\to C[n]$ is a projective scheme in Section
\ref{sec:projectivity}. In Section \ref{sec:fibres} we study its fibres.
Finally, in Section \ref{sec:linearization} we restrict to the case of a
bipartite oriented graph $\Gamma$ and in this situation equip $X[n]$ with a
particular linearization of the $G[n]$-action and compute the associated
Hilbert--Mumford invariants.

\subsection{Construction of the family $X[n]\to C[n]$}\label{sec:expanded-summary}

In this section we summarize work from Li \cite{li-2001, li-2013}, Wu
\cite{wu-2007} and Li--Wu \cite{LW-2011} on expanded degenerations.

\subsubsection{Setup}

Let $C$ be a smooth curve with a distinguished point $0\in C$. Following the
terminology of Li--Wu \cite[Def.~1.1]{LW-2011} we define:

\begin{definition}\label{def:simple}
	A \emph{simple degeneration} is a flat morphism $\pi \colon X\to C$
	from a smooth algebraic space $X$ to a $k$-smooth curve $C$ with a
	distinguished point $0\in C$, such that
	\begin{enumerate}[\upshape(i)]
		\item $\pi$ is smooth outside the central fibre $X_0 = \pi^{-1}(0)$ and
		\item the central fibre $X_0$ has normal crossing singularities and its
		      singular locus $D\subset X_0$ is smooth.
	\end{enumerate}
	We call a simple degeneration {\em strict} if all components of $X_0$ are
	smooth.
\end{definition}

In \'etale local coordinates, a simple degeneration $X\to C$ is thus of the form
$t=xy$. The main motivation for us is degenerations of K3 surfaces: Kulikov
models of type II are simple degenerations, whereas Kulikov models of type III
are not, because of triple intersections in the central fibre.

\begin{inputdata}
	The input data to Li's construction are
	\begin{itemize}
		\item a smooth base curve $C$ with a distinguished point $0\in C$ and an
		      \'etale morphism $t\colon C\to \AA^1$ with $t^{-1}(0) = \{0\}$,
		\item a strict simple degeneration $X\to C$ and
		\item an orientation of the dual graph $\Gamma$ of the special fibre $X_0$.
	\end{itemize}
\end{inputdata}

Let $G[n]\subset \SL(n+1)$ be the diagonal maximal torus and let $C[n]$ be the
fibre product $C\times_{\AA^1}\AA^{n+1}$ with respect to $t\colon C\to \AA^1$
and the multiplication morphism $\AA^{n+1}\to \AA^1$. Then $G[n]$ acts naturally
on $\AA^{n+1}$ such that $\AA^{n+1}\to \AA^1$ is invariant. Hence there are
induced actions on $C[n]$ and on $X\times_C C[n]$.

\begin{outputdata}
	The output of Li's construction is an explicit small $G[n]$-equivariant
	resolution $X[n]$ of the fibre product $X\times_C C[n]$.
\end{outputdata}

\begin{remark}
	\leavevmode
	\begin{itemize}
		\item The logic here is not axiomatic; rather we give the explicit
		      construction of $X[n]$ first and study its properties afterwards.
		\item Li's construction applies also to non-strict simple degenerations, but
				this requires an additional hypothesis and is more cumbersome to
				state. We briefly treat the non-strict case in Section \ref{sec:schemeness}.
		\item It suffices to treat $C=\AA^1$ since $X\times_C C[n] \iso
				X\times_{\AA^1}\AA^{n+1}$. We occationally take advantage of this,
				notably in Section \ref{sec:two-components}, but otherwise keep the
				base curve $C$ as this is notationally convenient.
	\end{itemize}
\end{remark}

\begin{remark}\label{rem:normalization}
	Li does not use the language of an orientation of the dual graph $\Gamma$.
	Instead, let $\nu\colon \widetilde{X_0} \to X_0$ be the normalization
	morphism. For each component $D$ of the singular locus of $X_0$ the inverse
	image $\nu^{-1}(D)$ is a disjoint union of two copies of $D$. Li fixes a
	labelling $\nu^{-1}(D) = D^+ \cup D^-$ for each such $D$, and in fact
	considers this labelling as part of the data defining a simple degeneration.
	This is equivalent to our orientation of $\Gamma$: viewing the nodes of
	$\Gamma$ as the components of $\widetilde{X_0}$, the corresponding
	orientation of the edge $[D]$ points from the (node corresponding to the)
	component containing $D^-$ to that containing $D^+$.
\end{remark}

\subsubsection{The basic case}\label{sec:two-components}

We first define $X[n]$ in the special case where $X_0 = Y_1 \cup Y_2$ has two
smooth components with irreducible intersection $D = Y_1 \cap Y_2$. The dual
graph $\Gamma$ consists of two nodes $[Y_i]$ connected by one edge $[D]$. Fix
one out of the two possible orientations and then relabel the two components
$Y_i$ if necessary so that $[D]$ points from $[Y_1]$ to $[Y_2]$:
\begin{equation*}
	[Y_1] \xrightarrow{[D]} [Y_2]
\end{equation*}
Recursively assume a small resolution $X[n-1]\to X\times_{\AA^1}\AA^n$ has been
constructed. View this first as a morphism over $\AA^n$ and then via the last
coordinate $t_n\colon \AA^n\to \AA^1$ as a morphism over $\AA^1$. Now pull back
along the multiplication morphism $m\colon \AA^2\to \AA^1$ to obtain a partial
resolution
\begin{equation}\label{eq:partial}
	X[n-1]\times_{\AA^1}\AA^2\to X\times_{\AA^1}\AA^{n+1}
\end{equation}
as in the diagram
\begin{equation*}
	\begin{tikzcd}[column sep=large]
		X[n-1]\times_{\AA^1}\AA^2 \ar[d] \ar[r]
			& X\times_{\AA^1}\AA^{n+1} \ar[d] \ar[r]
			& \AA^{n+1} \ar[d, "{(t_1,\dots,t_{n-1},m)}"'] \ar[r, "{(t_n, t_{n+1})}"]
			& \AA^2 \ar[d, "m"'] \\
		X[n-1] \ar[r] & X\times_{\AA^1}\AA^n \ar[r] & \AA^n \ar[r, "t_n"] & \AA^1
	\end{tikzcd}
\end{equation*}
where all squares are Cartesian.

\begin{proposition}\label{prop:two-components}
The following recursion defines a small $G[n]$-equivariant resolution $X[n]$ of
the fibre product $X\times_{\AA^1}\AA^{n+1}$ for all natural numbers $n$:
\begin{enumerate}[\upshape(i)]
	\item $X[0] = X$
	\item $X[n]$ is the blow-up of $X[n-1]\times_{\AA^1}\AA^2$ along the strict
	      transform of
	      \begin{equation*}
	      	Y_1\times V(t_{n+1}) \subset X \times_{\AA^1}\AA^{n+1}
	      \end{equation*}
			under the partial resolution \eqref{eq:partial}. Here the chosen
			orientation of the dual graph $\Gamma$ is used to order the components
			$Y_i$ of $X_0$ by the convention that the arrow $[D]$ has source
			$[Y_1]$ and target $[Y_2]$.
\end{enumerate}
\end{proposition}

\begin{proof}
	First consider $X[1]$. This is defined to be the blow-up of $X\times_{\AA^1}
	\AA^2$ along $Y_1\times V(t_2)$. In \'etale local coordinates $X\to \AA^1$ is
	$t=xy$ and the fibre product $X\times_{\AA^1} \AA^2$ is $t_1t_2=xy$. Blowing
	up either of the Weil divisors $x=t_2=0$ or $y=t_2=0$ yields a small
	resolution. Explicitly, let $U$ be an \'etale local neighbourhood in $X$ with
	coordinates $x$ and $y$ such that $Y_1$ is given by the equation $y=0$. The
	blow-up along $y=t_2=0$ is then the locus
	\begin{align*}
		ux &= t_1v\\
		vy &= t_2u
	\end{align*}
	in $U\times\AA^2\times\PP^1$, where $(u:v)$ are homogeneous coordinates on
	$\PP^1$. Thus $X[1]$ is indeed a small resolution.
	
	This basic construction is now repeated: view $X[n-1]$ as a family over
	$\AA^1$ via the last coordinate $t_n\colon \AA^n\to \AA^1$.
	
	\begin{quote}
		\textit{Claim:} 
		$X[n-1]\to \AA^1$ is a simple degeneration. Its central fibre has two
		components, the strict transforms $Y_i^{(n-1)}$ of $Y_i\times
		V(t_n)\subset X\times_{\AA^1}\AA^n$. Their intersection $D^{(n-1)}$ is the
		strict transform of $D\times V(t_n)$ and it is irreducible.
	\end{quote}
	
	Granted this, the dual graph of $\res{X[n-1]}{t_n=0}$ is identified with that
	of $X_0$ and inherits an orientation
	\begin{equation*}
		[Y^{(n-1)}_1] \xrightarrow{[D^{(n-1)}]} [Y^{(n-1)}_2].
	\end{equation*}
	Now the strict transform of $Y_1\times V(t_{n+1})$ in
	$X[n-1]\times_{\AA^1}\AA^2$ is precisely
	\begin{equation*}
		Y^{(n-1)}_1\times V(t_{n+1}) \subset X[n-1]\times_{\AA^1}\AA^2
	\end{equation*}
	(writing $(t_n,t_{n+1})$ for the coordinates on the last factor $\AA^2$).
	Thus the recipe says $X[n] = X[n-1][1]$. In view of this it suffices to
	verify the claim for $X[1]$ viewed as a family over $\AA^1$ via $t_2$: in
	fact under the blow-up
	\begin{equation*}
		X[1]\to X\times_{\AA^1} \AA^2
	\end{equation*}
	the strict transform of $Y_1\times V(t_2)$ is an isomorphic copy of it, the
	strict transform of $Y_2\times V(t_2)$ is its blow-up along $D\times
	\{(0,0)\}$, and they are normal crossing divisors intersecting along the
	strict transform of $D\times V(t_2)$, which is an isomorphic copy of it.
	These statements are readily verified with the help of the local equations
	for $X[1]$ given above.
\end{proof}

\subsubsection{The general case}\label{sec:general-case}
Return to the situation of an arbitrary strict simple degeneration $X\to C$,
where the central fibre $X_0\subset X$ and its singular locus are allowed to
have several components. Fix an orientation of the dual graph $\Gamma$. We
phrase the definition of $X[n]$ in somewhat informal language and expand the
precise meaning below.

\begin{definition}\label{def:expanded}
	Let $X[n]\to X\times_C C[n]$ be the small resolution obtained by applying
	Proposition \ref{prop:two-components} locally around each component of the
	singular locus of $X_0$.
\end{definition}

Explicitly:
\begin{itemize}
	\item For each component $D$ of the singular locus of $X_0$ there are unique
			components $Y_1$ and $Y_2$ of $X_0$ such that $D$ is a component of
			$Y_1\cap Y_2$. Use the chosen orientation of $\Gamma$ to distinguish
			the role of the two components $Y_i$, by labelling them so that the
			arrow $[D]$ in the oriented dual graph $\Gamma$ points from $[Y_1]$ to
			$[Y_2]$. \item Define $U\subset X$ to be the Zariski open subset whose
			complement is the union of all components of $X_0$ except $Y_1$ and
			$Y_2$ together with all components of $Y_1\cap Y_2$ except $D$. Thus
			$U$ is a Zariski open neighbourhood around $D$ such that $U_0$ has
			exactly two components $Y_i\cap U$ with intersection $D$.
	\item Apply Proposition \ref{prop:two-components} to obtain the small resolution
	      $U[n]\to U\times_C C[n]$.
	\item This small resolution is an isomorphism away from $D\times_C C[n]$, and
			the collection of all $U$ as $D$ varies covers $X$. Thus the various
			$U[n]$ glue to yield the global small resolution $X[n]\to X\times_C
			C[n]$.
\end{itemize}

\subsubsection{The group action}
We equip $X[n]$ with a $G[n]$-action such that $X[n]\to X\times_C C[n]$ is
equivariant. The target is here equipped with the action induced by the natural
action of $G[n]$ on $\AA^{n+1}$ and (hence) on $C[n]$.

Recall that an element in $G[n]$ is an $(n+1)$-tuple $(\sigma_1,\dots,
\sigma_{n+1})$ of elements in $\Gm$ such that $\prod_i \sigma_i = 1$. There is a
Cartesian diagram
\begin{equation*}
	\begin{tikzcd}
		G[n] \ar[r, "{\pr_{n,n+1}}"] \ar[d, "{m_{n,n+1}}"] & \Gm^2 \ar[d, "m"]\\
		G[n-1] \ar[r, "\pr_n"] & \Gm
	\end{tikzcd}
\end{equation*}
where $\pr_n$ and $\pr_{n,n+1}$ are projections onto the last and the last two
coordinates, $m$ is multiplication and $m_{n,n+1}$ multiplies together the last
two coordinates.

\begin{proposition}
	There is a unique $G[n]$-action on $X[n]$ such that
	\begin{equation*}
		X[n]\to X[n-1]\times_{\AA^1}\AA^2
	\end{equation*}
	is equivariant with respect to the natural action of $G[n]\iso
	G[n-1]\times_{\Gm}\Gm^2$ on the target.
\end{proposition}

\begin{proof}
	In view of the local nature of Definition \ref{def:expanded} it suffices to
	treat the basic situation in Proposition \ref{prop:two-components}. Then
	$X[n]\to X[n-1]\times_{\AA^1}\AA^2$ is the blow-up along the strict transform
	of $Y_1\times V(t_2)$, which is $G[n]$-invariant. Hence the action lifts
	uniquely.
\end{proof}

\subsubsection{Local equations}\label{sec:localcoordinates}
It is useful to have explicit equations for $X[n]$ in the case of the local
model
\begin{equation}\label{eq:etale-local}
	X=\Spec k[x,y,z,\dots] \xrightarrow{t=xy} C=\Spec k[t].
\end{equation}
Consider the product
\begin{equation*}
	(X\times \AA^{n+1}) \times \left(\PP^1\right)^n
	= \Spec k[x,y,z,\dots, t_1,\dots,t_{n+1}]\times\prod_i \Proj k[u_i, v_i]
\end{equation*}
and its subvariety $(X\times_{\AA^1} \AA^{n+1}) \times \left(\PP^1\right)^n$
defined by $xy=t_1t_2\cdots t_{n+1}$.

\begin{proposition}[{Wu \cite[\S 4.2]{wu-2007}}]\label{prop:localeq}
	Let $X\to C$ be the simple degeneration \eqref{eq:etale-local}, with dual
	graph oriented as $[V(y)] \to [V(x)]$. Then
	\begin{itemize}
		\item[(i)]
			$X[n]$ is the subvariety of $(X\times_{\AA^1} \AA^{n+1}) \times
			\left(\PP^1\right)^n$ defined by the equations
			\begin{align*}
			&u_1 x = v_1 t_1 \\
			&u_i v_{i-1} = v_i u_{i-1} t_i & &\left(1<i\le n\right) \\
			&v_n y = u_n t_{n+1}.
			\end{align*}
		\item[(ii)]
			The $G[n]$-action on $X[n] \subset (X\times_{\AA^1} \AA^{n+1}) \times
			(\PP^1)^n$ is the restriction of the action which is trivial on $X$,
			given by
			\begin{equation*}
				(t_1, \dots, t_{n+1}) \overset{\sigma}{\mapsto}
				(\sigma_1t_1,\dots,\sigma_{n+1}t_{n+1})
			\end{equation*}
			on $\AA^{n+1}$, and given by
			\begin{equation*}
				(u_i : v_i) \overset{\sigma}{\mapsto}
				(\sigma_1\sigma_2\cdots\sigma_i u_i:v_i)
			\end{equation*}
			on the $i$'th copy of $\PP^1$.
	\end{itemize}
\end{proposition}

This is straight forward to verify from the construction in Proposition
\ref{prop:two-components}.

\begin{remark}\label{rem:tau-coord}
	There is an isomorphism:
	\begin{align*}
		\Gm^n  &\iso G[n] \subset \Gm^{n+1}\\
		(\tau_1,\tau_2,\dots,\tau_n) &\mapsto (\tau_1, \tau_1^{-1} \tau_2,
		\tau_2^{-1} \tau_3, \dots, \tau_{n-1}^{-1}\tau_n, \tau_n^{-1})
	\end{align*}
	In $\tau$-coordinates, the action on the $\PP^1$-factors above is
	conveniently written as $(u_i:v_i) \mapsto (\tau_i u_i: v_i)$. (Note that Li
	\cite{li-2001} writes $\overline{\sigma}$ for our $\sigma$, and $\sigma$ for
	our $\tau$.)
\end{remark}

Let $(u_0:v_0) = (1:x)$ and $(u_{n+1}:v_{n+1})=(y:1)$, so that the equations in
Proposition \ref{prop:localeq} can be written uniformly as
\begin{equation}\label{eq:localeq}
	u_i v_{i-1} = v_i u_{i-1} t_i,\qquad \left(1\le i\le n+1\right).
\end{equation}
These local equations immediately lead to the explicit affine open cover given
by Li \cite[Lemma 1.2]{li-2001}: we have $X[n]=\Union_{k=1}^{n+1} W_k$ , where
\begin{equation*}
	W_k\colon
	\begin{cases}
		u_i\ne 0 & \text{for $i<k$}\\
		v_i\ne 0 & \text{for $i\ge k$.}
	\end{cases}
\end{equation*}
In each chart $W_k$, most of the equations \eqref{eq:localeq} result in
elimination of either $u_i/v_i$ or $v_i/u_i$, so that $W_k$ has coordinates
\begin{equation*}
	t_1,\dots, t_{n+1}, \tfrac{v_{k-1}}{u_{k-1}}, \tfrac{u_k}{v_k}
\end{equation*}
(together with $z, \dots$) subject to the single relation $t_k =
\tfrac{u_k}{v_k} \tfrac{v_{k-1}}{u_{k-1}}$. Each $W_k$ is clearly
$G[n]$-invariant, and the $G[n]$-action is
\begin{align*}
	(t_1, \dots, t_n) \overset{\sigma}{\mapsto} (\sigma_1t_1,\dots,\sigma_nt_n)\\
	\left(\tfrac{v_{k-1}}{u_{k-1}}, \tfrac{u_k}{v_k}\right) \overset{\sigma}{\mapsto}
	\left(\tau_{k-1}^{-1}\tfrac{v_{k-1}}{u_{k-1}}, \tau_k\tfrac{u_k}{v_k}\right)
\end{align*}
on points; here $\tau_i = \sigma_1\cdots\sigma_i$ as in Remark
\ref{rem:tau-coord}.

\subsection{Projectivity criteria}\label{sec:projectivity}

\subsubsection{Preliminaries}

In the arguments that follow, we shall frequently make use of the exactness of
\begin{equation*}
	{\textstyle\bigoplus_i} \ZZ D_i \to \Pic(Y)
	\to \Pic(Y\setminus\union_i D_i) \to 0
\end{equation*}
whenever $\{D_i\}$ is a finite set of effective prime divisors in a nonsingular
variety $Y$ (see e.g.\ \cite[Prop.~II.6.5]{hartshorne-1977}). Whenever $P\to Y$
is a principal $\Gm$-bundle, with associated line bundle $L$, we also have an
exact sequence
\begin{equation*}
	\ZZ \xrightarrow{c_1(L)} \Pic(Y)\to \Pic(P) \to 0.
\end{equation*}
This follows by applying the first short exact sequence to the line bundle $L$
and its $0$-section $H_0$, together with the fact that pullback defines an
isomorphism $\Pic(Y) \to \Pic(L)$ which identifies $c_1(L)$ with $H_0$.

\subsubsection{Criterion for $X[n]$ to be a scheme.}\label{sec:schemeness}

Let $X\to C$ be a non-strict simple degeneration. Thus the central fibre $X_0$
contains at least one singular (``self intersecting'') component $Y$. Let
$D\subset Y$ be a component of the singular locus of such a component.

We can no longer apply Proposition \ref{prop:two-components} Zariski locally
around $D$, but we still can do so in an \'etale local sense provided we can
distinguish \'etale locally around $D$ between the two branches of $Y$ meeting
there. This blends well with Li's use of the normalization morphism $\nu\colon
\widetilde{X}_0 \to X_0$ explained in Remark \ref{rem:normalization}: we define
a non-strict simple degeneration as in Definition \ref{def:simple}, but replace
strictness by the following condition on each component $D$ of the singular
locus of $X_0$:
\begin{quote}
	Assume the preimage $\nu^{-1}(D)$ is a disjoint union of two
	copies of $D$.
\end{quote}
Under this condition Definition \ref{def:expanded} of $X[n]\to C[n]$ applies to
non-strict simple degenerations when the word ``locally'' is interpreted in the
\'etale topology.

We refrain from giving further details as our aim here is just to point out that
in the non-strict situation $X[n]$ is an algebraic space but never a scheme.

\begin{proposition}\label{prop:scheme-criterion}
	Let $X\to C$ be a simple degeneration where $X$ is a scheme. Then, for all
	$n>0$, the algebraic space $X[n]$ is a scheme if and only if $X \to C$ is
	strict, i.e.\  if and only if the graph $\Gamma(X_0)$ contains no loops.
\end{proposition}

\begin{proof}
	If $X\to C$ is strict, then $X[n]$ is a scheme by construction: it is
	recursively defined as a resolution $X[n]\to X[n-1]\times_{\AA^1}\AA^2$ given
	Zariski locally on the target by blowing up Weil divisors.

	Conversely let $X\to C$ be non-strict. We reduce to $n=1$: let $\AA^2\to
	\AA^{n+1}$ be the map $(t_1,t_2)\mapsto(t_1,t_2,1,\dots,1)$ and let $C[1]\to
	C[n]$ be the induced map. Then there is a Cartesian diagram (see\
	e.g.~\cite[2.14 and 2.15]{li-2013})
	\begin{equation*}
		\begin{tikzcd}
			X[1] \ar[r] \ar[d] &  X[n] \ar[d] \\
			C[1] \ar[r] & C[n]
		\end{tikzcd}
	\end{equation*}
	so that if $X[1]$ fails to be a scheme, then so does $X[n]$.
	
	As $X_0$ is non-strict there exists a singular component $Y\subset
	X_0$. Fix a singular point $P\in Y$. The inverse image of $(P; 0,0)$ by
	$X[1]\to X\times_{\AA^1}\AA^2$ is a $\PP^1$. If $X[1]$ is a scheme, hence a
	nonsingular variety, then there exists an effective divisor $H\subset X[1]$
	intersecting this $\PP^1$ in a positive number of points. In particular the
	corresponding line bundle $\sheaf{L}=\OO_{X[1]}(H)$ has nontrivial
	restriction to $\PP^1$. Choose a Zariski open neighbourhood $U\subset X$ of
	$P$, such that $U$  does not intersect any other component of $X_0$ besides
	$Y$. Then $U[1]$ is a Zariski open neighbourhood of $\PP^1\subset X[1]$. The
	inverse image by $U[1]\to \AA^2$ of each of the coordinate axes $V(t_i)
	\subset \AA^2$ is a principal prime divisor $D_i \subset U[1]$. Thus
	$\Pic(U[1]) \iso \Pic(U[1]\setminus(D_1\union D_2))$ by restriction. Also,
	the fibre $U_0\subset U$ over $0\in \AA^1$ is a principal prime divisor, so
	$\Pic(U) \iso \Pic(U\setminus U_0)$ by restriction. But $U[1]\setminus
	D_1\union D_2$ is a principal $\Gm$-bundle over $U\setminus U_0$, being a
	pullback of the multiplication map $\AA^2\setminus \{t_1t_2=0\} \to
	\AA^1\setminus \{0\}$. Thus pullback $\pi^*\colon \Pic(U) \to \Pic(U[1])$ is
	surjective, and thus there is a line bundle $\sheaf{M}$ on $U$ such that
	$\pi^*(\sheaf{M})\iso\res{\sheaf{L}}{U[1]}$. But the restriction of
	$\pi^*(\sheaf{M})$ to $\PP^1$ is trivial.
\end{proof}

\subsubsection{Criterion for $X[n]\to C[n]$ to be projective}

\begin{proposition}\label{prop:projectivity-criterion}
	Let $X\to C$ be a projective strict simple degeneration with oriented dual
	graph $\Gamma(X_0)$. Then, for each $n>0$, the morphism $X[n]\to C[n]$ is
	projective if and only if $\Gamma(X_0)$ contains no directed cycles.
\end{proposition}

\begin{proof}[Proof that no cycles $\implies$ projective]
	By induction we can assume that $X[n-1]\to C[n-1]$ is projective. Then
	$X[n-1]\times_{C[n-1]} C[n] \to C[n]$ is projective as well. It thus suffices
	to show that the resolution $X[n]$ of $X[n-1]\times_{C[n-1]} C[n] \iso
	X[n-1]\times_{\AA^1}\AA^2$ is \emph{globally} given by blowing up Weil
	divisors and their strict transforms in a certain order. The required order
	will be dictated by the oriented graph $\Gamma(X_0)$.

	As $\Gamma(X_0)$ has no directed cycles, the components $Y$ of $X_0$ are
	partially ordered by the rule $Y \le Y'$ if there is a directed path
	from the node $[Y]$ to $[Y']$ in $\Gamma(X_0)$. We first claim that the
	resolution
	\begin{equation*}
		X[1] \to X\times_{\AA^1}\AA^2
	\end{equation*}
	is the blow-up of all the Weil divisors $Y \times V(t_2)$ and their strict
	transforms, in increasing order with respect to the partial order of the
	components $Y$ just introduced. In fact, if $D\subset Y\cap Y'$ is a
	component corresponding to an arrow in the direction
	\begin{equation*}
		[Y] \xrightarrow{[D]} [Y']
	\end{equation*}
	then the construction of $X[1]$ in Definition \ref{def:expanded} and
	Proposition \ref{prop:two-components} instructs us to blow up along $Y\times
	V(t_2)$ in a neighbourhood of $D\times\{0,0\}$. Thus globally blowing up
	along $Y\times V(t_2)$ has the required effect there. Moreover, having
	resolved the singularity at $D\times \{(0,0)\}$, the strict transform of the
	Weil divisor $Y'\times V(t_2)$ is now Cartier over $D\times_{\AA^1}\AA^2$,
	so a later blow-up in the given partial order has no effect over those loci
	$D\times \{(0,0)\}$ already resolved. Lastly, if $Y$ and $Y'$ are unrelated
	in the partial order they are disjoint, so the blow-up order is irrelevant.
	This proves the claim for $X[1]$.

	$X[1]$ viewed as a family over $\AA^1$ via second projection $\AA^2\to \AA^1$
	is again simple with dual graph $\Gamma(\res{X[1]}{t_2=0})$ canonically
	isomorphic to $\Gamma(X_0)$. More precisely, for $Y$ and $D$ running through
	the components of $X_0$ and the components of its singular locus
	respectively, let $Y^{(1)}$ and $D^{(1)}$ denote the strict transforms of
	$Y\times V(t_2)$ and of $D\times V(t_2)$ by the resolution $X[1]\to
	X\times_{\AA^1}\AA^2$. Then $Y^{(1)}$ are precisely the components of
	$\res{X[1]}{t_2=0}$ and $D^{(1)}$ are precisely the components of its
	singular locus. This follows by applying the \emph{claim} in the proof of
	Proposition \ref{prop:two-components} to Zariski open neighbourhoods $U$
	around $D$. By induction the resolution
	\begin{equation*}
		X[n] = X[n-1][1] \to X[n-1]\times_{\AA^1}\AA^2
	\end{equation*}
	is thus the composition of the blow-ups along the strict transforms of
	$Y^{(n-1)}\times V(t_2)$ in increasing order with respect to the partial
	order of the components $Y$. It is thus projective.
\end{proof}

\begin{proof}[Proof that cycle $\implies$ not projective]
	As in the proof of Proposition \ref{prop:scheme-criterion}, we reduce to
	$n=1$ by pullback along $C[1]\to C[n]$.
	
	Let $X_0=\Union Y_i$ and $D=\Union D_u$ be the decompositions of the special
	fibre and its singular locus into irreducible components. Fix a point $P_u$
	on each $D_u$. The fibre of $X[1]\to X\times_{\AA^1}\AA^2$ over each $(P_u;
	0,0)$ is a $\PP^1$, which we denote $\PP^1_u$. If $X[1]\to C[1]$ is
	projective, there exists a relatively ample line bundle $\sheaf{L}$ on
	$X[1]$. Thus $\sheaf{L}$ restricts to an ample line bundle on the fibres of
	$X[1]\to C[1]$; in particular it restricts to a line bundle of positive
	degree on each $\PP^1_u$.
	
	For $m=1,2$, let $H_{im}\subset X[1]$ be the divisor obtained as strict
	transform of $Y_i\times_{\AA^1} V(t_m)\subset X\times_{\AA^1}\AA^2$. Arguing
	as in the proof of Proposition \ref{prop:scheme-criterion}, we arrive at the
	diagram
	\begin{equation*}
		\begin{tikzcd}
			{\textstyle \bigoplus} \ZZ H_{im} \ar[r] & \Pic(X[1]) \ar[r]
				& \Pic(X[1]\setminus \union H_{im}) \ar[r] & 0\\
			{\textstyle \bigoplus} \ZZ Y_i \ar[r] & \Pic(X) \ar[u, "\pi^*"] \ar[r]
				& \Pic(X\setminus X_0) \ar[r] \ar[u, two heads] & 0
		\end{tikzcd}
	\end{equation*}
	(where the rightmost vertical map is surjective by the principal
	$\Gm$-bundle argument). Hence $\sheaf{L}$ can be written
	$\OO_{X[1]}(\sum_{km} n_{km} H_{km})\tensor \pi^*(\sheaf{M})$ for some line
	bundle $\sheaf{M}$ on $X$. But the restriction of $\pi^*(\sheaf{M})$ to each
	$\PP^1_u$ is trivial, hence also $\sum_{km} n_{km} H_{km}$ has positive
	degree on each $\PP^1_u$. We shall show that this imposes conditions on the
	coefficients $n_{km}$ that are incompatible with the presence of a cycle in
	$\Gamma(X_0)$.
	
	For each $D_u$, there is a corresponding arrow $[Y_j] \to [Y_i]$ in
	$\Gamma(X_0)$. Replace $X$ with a Zariski local neighbourhood of $D_u$ such
	that $X_0$ just consists of the two components $Y_i$ and $Y_j$, with
	irreducible intersection $D_u$. This has the effect of replacing $X[1]$ with
	a Zariski open neighbourhood of $\PP^1_u$. Then $X[1]$ is the blow-up of
	$X\times_{\AA^1}\AA^2$ along the Weil divisor $Y_j\times_{\AA^1} V(t_2)$. One
	can check, e.g.\ by a computation in local coordinates, that the total and
	strict transforms of $Y_j\times V(t_2)$ agree. Hence $H_{j,2}$, viewed as the
	inverse image of the blow-up centre, restricts to $\OO_{\PP^1_u}(-1)$ on
	$\PP^1$. Locally around $\PP^1_u$, the divisors $H_{i,1} + H_{j,1}$,
	$H_{i,2}+H_{j,2}$, $H_{i,1}+H_{i,2}$ and $H_{j,1}+H_{j,2}$ are all principal,
	given by $t_1=0$, $t_2=0$, a local equation for $Y_i$ and a local equation
	for $Y_j$, respectively. So we have
	\begin{align*}
		\OO_{\PP^1_u}(H_{i,1}) &= \OO_{\PP^1_u}(-1) &
		\OO_{\PP^1_u}(H_{i,2}) &= \OO_{\PP^1_u}(1) \\
		\OO_{\PP^1_u}(H_{j,1}) &= \OO_{\PP^1_u}(1) &
		\OO_{\PP^1_u}(H_{j,2}) &= \OO_{\PP^1_u}(-1)
	\end{align*}
	whereas all other $\OO_{\PP^1_u}(H_{k,m})$ are trivial, for $m=1,2$ and $k$
	anything but $i$ and $j$. Thus, the condition for $\sum_{km}n_{km} H_{km}$ to
	have positive degree on $\PP^1_u$ is
	\begin{equation*}
		(n_{j,1} -  n_{j,2}) + (n_{i,2} -n_{i,1})> 0.
	\end{equation*}
	Now label the nodes $[Y_j]$ in a directed loop as $j=1,\dots,r$. Then
	\begin{gather*}
		(n_{1,1} - n_{1,2}) + (n_{2,2} - n_{2,1}) > 0\\
		(n_{2,1} - n_{2,2}) + (n_{3,2} - n_{3,1}) > 0\\
		(n_{3,1} - n_{3,2}) + (n_{4,2} - n_{4,1}) > 0\\
		\vdots\\
		(n_{r,1} - n_{r,2}) + (n_{1,2} - n_{1,1}) > 0
	\end{gather*}
	and the sum of the left hand sides is zero; this is the required
	contradiction.
\end{proof}

\subsubsection{Inversion of orientation}\label{sec:inversion}

The expanded degeneration $X[n] \to C[n]$ depends on a choice of orientation of
the dual graph $\Gamma(X_0)$. We observe that the effect of reversing the
orientation, i.e.\ reversing the direction of all arrows, is only to permute the
coordinates in $C[n]$:

\begin{proposition}\label{prop:inversion}
	Let $X[n] \to C[n]$ and $X[n]'\to C[n]$ be the two expanded degenerations
	associated with opposite orientations of the dual graph $\Gamma(X_0)$.
	Then there is a $G[n]$-equivariant isomorphism $X[n] \iso X[n]'$ covering the
	involution $\rho$ of $C[n] = C\times_{\AA^1}\AA^{n+1}$ induced by
	\begin{equation*}
		\AA^{n+1} \to \AA^{n+1},\quad (t_1,t_2,\dots,t_{n+1})
		\mapsto (t_{n+1},t_n,\dots,t_1).
	\end{equation*}
\end{proposition}

\begin{proof}
	$X[n]$ and $X[n]'$ are both resolutions of $X\times_C C[n]$, so there is a
	unique birational map $\phi$ making the following diagram commute:
	\begin{equation*}
		\begin{tikzcd}
			X[n] \ar[r, dashed, "\phi"] \ar[d] & X[n]' \ar[d]\\
			X\times_C C[n] \ar[r, "1_X\times \rho"] & X\times_C C[n]
		\end{tikzcd}
	\end{equation*}
	We claim that $\phi$ is in fact biregular. This is an \'etale local claim
	over $X$, so it suffices to verify that $\phi$ is biregular in the situation
	of the local equations in Proposition \ref{prop:localeq} (i). This is
	immediate, since reversal of the orientation amounts to interchanging the
	roles of $x$ and $y$ in these equations. It is clear that $\phi$ is
	equivariant.
\end{proof}

\subsection{The fibres of $X[n] \to C[n]$}\label{sec:fibres}

Fix a strict simple degeneration $X\to C$ and an orientation of the dual graph
$\Gamma$. We shall introduce notation describing the fibres of $X[n] \to C[n]$
and how they are smoothed as coordinates in $\AA^{n+1}$ move from zero to
nonzero.

\subsubsection{Expanded graphs}

Let $I\subset [n+1]$ be a subset and let $\AA^{n+1}_I\subset \AA^{n+1}$ be the
locus where all the coordinates $t_i$ vanish for $i\in I$. Let $C[n]_I =
C\times_{\AA^1} \AA^{n+1}_I$ and let $X[n]_I \to C[n]_I$ be the restriction of
$X[n]\to C[n]$. Let $\Gamma$ be the dual graph of $X_0$ equipped with an
orientation.

In view of future applications we will use the following notation. If
$I\subset [n+1]$ is a non-empty subset, then we denote its elements by
\begin{equation*}
I=\{a_1, \ldots, a_r \}, \quad a_1 < a_2 < \ldots  < a_r.
\end{equation*}
We construct an oriented graph $\Gamma_I$
(associated to $\Gamma$) by replacing each arrow
\begin{equation*}
	\bullet \to \bullet
\end{equation*}
in $\Gamma$ with $|I|$ arrows labelled by $I$ in ascending order in the
direction of the arrow:
\begin{equation}\label{eq:quiver-D}
	\bullet \xrightarrow{a_1} \circ \xrightarrow{a_2} \circ \to \cdots
	\xrightarrow{a_r} \bullet
\end{equation}
It is useful to colour the old nodes black and the new ones white --- so the
valence of any white node is $2$, and the valence of any black node is
unchanged from $\Gamma$. Label the black nodes $[Y_I]$, where $[Y]$ is the
corresponding node in $\Gamma$. Label the white nodes $\Delta^{D,a_i}_I$, where
$a_i$ is the incoming arrow and $[D]$ is the corresponding arrow in $\Gamma$.
We frequently suppress $D$ and write $\Delta^{a_i}_I = \Delta^{D,{a_i}}_I$.

\subsubsection{Components of $X[n]_I$}

When $I\subset J$ are two non-empty subsets of $[n+1]$ we may view $\Gamma_I$
as constructed from $\Gamma_J$ by deleting all arrows labelled by $J\setminus
I$, and identifying the nodes at the ends of each deleted arrow. Thus the set of
nodes in $\Gamma_I$ is a quotient of the set of nodes in $\Gamma_J$, and we let
\begin{equation*}
	q = q_{J,I}\colon \Gamma_J \to \Gamma_I
\end{equation*}
denote the quotient map on nodes (it is not defined on arrows, despite the
notation).

\begin{proposition}\label{prop:components}
	Let $X\to C$ be a strict simple degeneration with oriented dual graph
	$\Gamma(X_0)$ and let $I\subset [n+1]$ be non-empty.
	\begin{enumerate}[a)]
		\item
			$X[n]_I$ is a union of nonsingular components with normal crossings
			without triple intersections, i.e.\ it is \'etale locally isomorphic to
			the union of two hyperplanes in affine space. Furthermore, each
			component is flat over $C[n]_I$ and is a simple degeneration over the
			pointed curve $(\AA^1, 0)$, via any coordinate $t_i\colon \AA^{n+1}_I
			\to \AA^1$ for $i\not\in I$.
		\item
			There is a natural isomorphism between $\Gamma_I$ and the dual graph of
			$X[n]_I$, uniquely determined by the following: each (black) node
			$[Y_I]$ in $\Gamma_I$ corresponds to a component $Y_I\subset X[n]_I$
			wich is mapped birationally onto $Y\times_C C[n]_I$ by the natural
			birational map $X[n]\to X\times_C C[n]$, whereas the (white) nodes
			$[\Delta^{D, a_i}_I]$ correspond to components $\Delta^{D,a_i}_I \subset
			X[n]_I$ which are contracted onto $D\times_C C[n]_I$.
		\item
			When $V$ is a component of $X[n]_I$ and $I\subset J$ are non-empty, the
			intersection $V\intsct X[n]_J$ is the union of all components $W$ of
			$X[n]_J$ such that $q([W]) = [V]$.
	\end{enumerate}
\end{proposition}

The Proposition can be seen as a detailed version of \cite[Lemma
2.2]{LW-2011} by Li--Wu and we only sketch a proof.

\begin{figure}
\centering
\input{expanded.pspdftex}
\caption{$X[1]$ over $\AA^2$}
\label{fig:X[1]}
\end{figure}
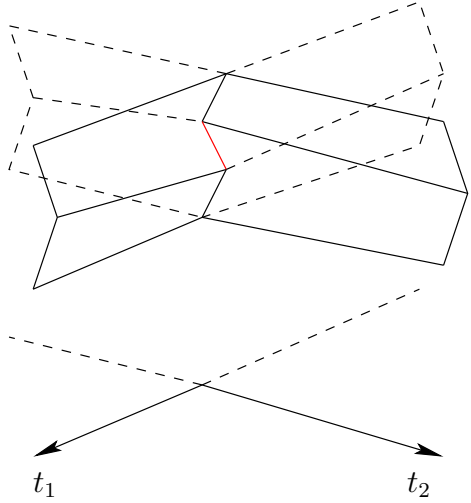

\begin{proof}
	The base curve $C$ plays no role so we let $C=\AA^1$. Argueing via an
	appropriate Zariski open cover of $X$ it suffices to treat the basic
	situation with central fibre $X_0=Y\cup Y'$ having two irreducible components
	and irreducible intersection $D=Y\cap Y'$. Label the components such that the
	arrow $[D]$ in $\Gamma(X_0)$ points from $[Y']$ to $[Y]$.
	
	Firstly the small resolution $\pi\colon X[1]\to X\times_{\AA^1}\AA^2$ is the
	blow-up along the Weil divisor $Y'\times (\AA^1\times\{0\})$. Let $E\subset
	X[1]$ denote the inverse image of $D\times\{(0,0)\}$. It is straight forward
	to verify the following by computing the blow-up explicitly:
	
	\begin{itemize}
		\item
			The restriction of $X[1]$ to $\{0\}\times\AA^1\subset\AA^2$ is a normal
			crossing union $X[1]_{\{1\}} = Y'_{\{1\}} \union Y_{\{1\}}$, where
			$\pi$ restricts to an isomorphism $Y'_{\{1\}}\to Y'\times
			(\{0\}\times\AA^1)$ and a blow-up $Y_{\{1\}} \to Y\times
			(\{0\}\times\AA^1)$ along $D\times \{(0,0)\}$. The exceptional divisor
			of the blow-up is $E\subset Y_{\{1\}}$. The intersection
			$Y'_{\{1\}}\intsct Y_{\{1\}}$ maps isomorphically to
			$D\times(\{0\}\times \AA^1)$. 
		\item
			The restriction of $X[1]$ to $\AA^1\times \{0\}\subset\AA^2$ is similar
			with the roles of $Y$ and $Y'$ interchanged.
		\item
			The restriction of $X[1]$ to $(0,0)\in \AA^2$ is a normal crossing
			union $X[1]_{\{1,2\}} = Y'_{\{1,2\}} \union \Delta^1_{\{1,2\}} \union
			Y_{\{1,2\}}$, where $\Delta^1_{\{1,2\}} = E$ and $\pi$ restricts to
			isomorphisms $Y'_{\{1,2\}} \to Y'\times \{(0,0)\}$ and $Y_{\{1,2\}} \to
			Y\times \{(0,0)\}$. Via these identifications, $Y'_{\{1,2\}}\intsct E$
			is $D\subset Y'$ and $Y_{\{1,2\}}\intsct E$ is $D\subset Y$, whereas
			$Y'_{\{1,2\}}$ and $Y_{\{1,2\}}$ are disjoint.
	\end{itemize}
	
	The proposition follows for $X[1]$.
	
	Before continuing it is useful to observe that by Proposition
	\ref{prop:inversion} inverting the orientation of $\Gamma(X_0)$ has the same
	effect as interchanging the coordinates on $\AA^2$. Thus $X[1]$ can equally
	well be obtained by blowing up $X\times_{\AA^1}\AA^2$ along $Y\times
	(\{0\}\times\AA^1)$.
	
	Inductively assume the proposition holds for $X[n-1]$. Let $I\subset [n+1]$
	be an index set containing neither $n$ nor $n+1$ and denote by $\overline{I}$
	the same set considered as a subset of $[n]$.
	
	\textit{Claim:} There is a fibre diagram
	\begin{equation*}
		\begin{tikzcd}
			X[n]_I \ar[r, "\pi_I"] & X[n-1]_{\overline{I}}\times_{\AA^1}\AA^2 \\
			Y_I \ar[r]\ar[u, hook] & Y_{\overline{I}}\times_{\AA^1}\AA^2\ar[u,hook]
		\end{tikzcd}
	\end{equation*}
	where
	\begin{itemize}
		\item $\pi_I$ is the restriction of $\pi$ to $\AA^{n+1}_I$
		\item $\pi_I$ is an isomorphism outisde $Y_I$
		\item $Y_I = Y_{\overline{I}}[1]$ viewing $Y_{\overline{I}}$ as a simple
				degeneration over $\AA^1$ via $t_n\colon \AA^n \to \AA^1$ with dual
				graph oriented such that its unique arrow points towards the node
				$[Y_{\overline{I}\union \{n\}}]$.
	\end{itemize}
	To verify the claim, take advantage of the observation preceeding it to view
	$\pi$ as the blow-up of $X[n-1]\times_{\AA^1}\AA^2$ along the Weil divisor
	$Y_{\{n\}} \times (\{0\}\times \AA^1)$. This blow-up centre restricts to
	\begin{equation*}
		\left(Y_{\{n\}}\times(\{0\}\times \AA^1)\right)\intsct
		\left(X[n-1]_{\overline{I}}\times_{\AA^1}\AA^2\right)
		= Y_{\overline{I}\union\{n\}} \times (\{0\}\times \AA^1).
	\end{equation*}
	by part (c) for $X[n-1]$. By a local computation one checks that the total
	and strict transforms of $Y_{\overline{I}}\times_{\AA^1}\AA^2$ agree, so that
	$Y_I$ is in fact its blow-up along
	$Y_{\overline{I}\cup\{n\}}\times(\{0\}\times\AA^1)$. Again this is
	$Y_{\overline{I}}[1]$. This proves the claim.
	
	By the claim one can read off the components of $X[n]_I$ inductively from
	$X[n-1]_{\overline{I}}$. Applying the $n=1$ case to $Y_I =
	Y_{\overline{I}}[1]$ one also obtains the components of
	$X[n]_{I\union\{n\}}$, $X[n]_{I\union\{n+1\}}$ and $X[n]_{I\union\{n,n+1\}}$.
	The verification of the proposition from the claim and the $n=1$ case is then
	reduced to a book keeping exercise we refrain from writing out.
\end{proof}

Figure \ref{fig:X[1]} depicts $X[1]$ over the coordinate axes in $\AA^2$. With
notation as in the proof the uppermost components are $Y'_{\{1\}}$ and
$Y'_{\{2\}}$ whereas the lowermost components are $Y_{\{1\}}$ and $Y_{\{2\}}$.
The fibre over the origin has the additional component $\Delta^1_{\{1,2\}}$.

\subsubsection{In local coordinates}

Consider the \'etale local situation from Section \ref{sec:localcoordinates},
where $X[n]$ has the explicit open affine cover $\{W_k\}$. Let $W_{k,I} = W_k
\intsct X[n]_I$. Let $(i,j)$ be a pair of consecutive elements in $I \union \{0,
n+2\}$. As one immediately verifies, each $\Delta^i_I$ is given by the local
expressions
\begin{align*}
	\Delta^i_I \intsct W_{i,I} &= V\left( \tfrac{v_{i-1}}{u_{i-1}} \right), \\
	\Delta^i_I \intsct W_{k,I} &= W_{k,I} &  &\text{for $i<k<j$}, \\
	\Delta^i_I \intsct W_{j,I} &= V\left( \tfrac{u_j}{v_j} \right).
\end{align*}
These expressions include $Y'_I$ and $Y_I$ for $Y'=V(y)$ and $Y=V(x)$ as the
extremal cases $i=0$ and $i=\max I$.

\subsection{Linearization}\label{sec:linearization}

In this section we shall assume that $X\to C$ is a \emph{projective} simple
degeneration, and we moreover assume that the dual graph $\Gamma(X_0)$ is
equipped with a \emph{bipartite} orientation (see below for a formal
definition). The main aim of this section is to exhibit a particular
$G[n]$-linearized line bundle on $X[n]$, which will then be used for our
application of GIT to the relative Hilbert scheme of $X[n]\to C[n]$ in Section
\ref{sec:GIT-analysis}. The choice of linearization we make is not obvious. The
one we have found has the advantage that it gives a well behaved semi-stable
locus in the Hilbert scheme. The bipartite condition is indeed a crucial
condition as we shall see in Section \ref{sec:GIT-analysis}, in particular
Example \ref{ex:bipartite}.

\subsubsection{\'Etale functoriality}

As preparation, we observe that the construction $X \mapsto X[n]$ is functorial
with respect to \'etale maps.

\begin{proposition}\label{prop:functoriality}
	Let $X\to C$ be a strict simple degeneration and let $f\colon X' \to X$ be an
	\'etale morphism. Orient the dual graphs such that the induced map
	$\Gamma(X'_0) \to \Gamma(X_0)$ is orientation preserving. Then there are
	induced \'etale morphisms $f[n]\colon X'[n] \to X[n]$ for all $n$, making the
	following diagram Cartesian:
	\begin{equation*}
		\begin{tikzcd}
			X'[n] \ar[d] \ar[r, "{f[n]}"] & X[n] \ar[d] \\
			X' \ar[r, "f"] & X.
		\end{tikzcd}
	\end{equation*}
\end{proposition}

\begin{proof}
	Observe that for each component $Y' \subset X'_0$, the image $f(Y')$ is dense
	in some component $Y\subset X_0$. Similarly, if $D'\subset X'_0$ is a
	component of the singular locus, the image $f(D')$ is dense in some component
	$D$ of the singular locus of $X_0$. This defines the map $\Gamma(X'_0) \to
	\Gamma(X_0)$ on vertices and edges, respectively.
	
	Let $D\subset X_0$ be a component of the singular locus and let $U\subset X$
	be a Zariski open neighbourhood of $D$, such that $U_0 = U\intsct X_0$ has
	two components $Y_1$ and $Y_2$ with $D=Y_1\intsct Y_2$.  Order the two
	components such that the arrow $[D]$ in $\Gamma(X_0)$ points from $[Y_1]$ to
	$[Y_2]$. Then $U[1] \to U\times_{\AA^1}\AA^2$ is the blow-up along
	$Y_1\times V(t_2)$.
	
	Let $U' = f^{-1}(U)$. It is a Zariski open subset of $X'$, and $U'_0 = U'
	\intsct X'_0$ has the following structure: it has a number (possibly zero) of
	components $Y'_{1,i}$ mapping to $Y_1$, a number of components $Y'_{2,i}$
	mapping to $Y_2$, the only non-empty intersections are of the form $Y'_{1,i}
	\intsct Y'_{2,i}$, and all components of $Y'_{1,i}\intsct Y'_{2,i}$ map to $D$. We
	abuse notation and write $[Y'_{j,i}]$ for the vertex in $\Gamma(X'_0)$
	corresponding to the closure of $Y'_{j,i}$. Since the map on oriented graphs
	respects the orientation, all arrows in $\Gamma(X'_0)$ point from
	$[Y'_{1,i}]$ to $[Y'_{2,i}]$. Thus $U'[1] \subset X'[1]$ is obtained by
	blowing up all $Y'_{1,i}\times V(t_2) \subset U'\times_{\AA^1} \AA^2$, and
	as the $Y'_{1,i}$'s are disjoint, they may be blown up simultaneously. As
	$f$ is \'etale, and hence flat, blow-up commutes with base change in the
	sense that the diagram
	\begin{equation*}
		\begin{tikzcd}
			U'[1] \ar[r]\ar[d] & U[1] \ar[d]\\
			U'\times_C C[1] \ar[r] & U\times_C C[1]
		\end{tikzcd}
	\end{equation*}
	is Cartesian. The topmost arrow defines $f[1]$ over $U'[1]$. Cover $X$ by
	Zariski open neighbourhoods $U$ of this form to define $f[1]$ everywhere.
	Being a pullback of the \'etale map $f$, the map $f[1]$ is also \'etale.
	
	In view of the recursive construction, the procedure may be repeated: having
	defined the \'etale morphism $f[n-1]\colon X'[n-1] \to X[n-1]$, inducing an
	orientation preserving map on the oriented graphs of the respective fibres
	over $t_n=0$, cover $X[n-1]$ by Zariski opens $U\subset X[n-1]$ as before and
	let $U' = f[n-1]^{-1}(U)$. Let $\widetilde{U}\subset X[n]$ be the inverse
	image of $U\times_{\AA^1}\AA^2$ by the resolution $X[n] \to
	X[n-1]\times_{\AA^1}\AA^2$ and let $\widetilde{U'}\subset X'[n]$ be defined
	analogously. The Cartesian diagram
	\begin{equation*}
		\begin{tikzcd}
			\widetilde{U'} \ar[r]\ar[d] & \widetilde{U}\ar[d] \\
			U'\times_{\AA^1} \AA^2 \ar[r] & U\times_{\AA^1} \AA^2
		\end{tikzcd}
	\end{equation*}
	defines $f[n]$ over $\widetilde{U'}$.
\end{proof}

\begin{remark}\label{rem:induced-orientation}
	In the notation of Proposition \ref{prop:functoriality}, once an orientation
	on $\Gamma(X_0)$ has been chosen, there is a unique orientation on
	$\Gamma(X'_0)$ making the map of graphs $\Gamma(X'_0) \to \Gamma(X_0)$
	orientation preserving.
\end{remark}

\begin{remark}\label{rem:equivariance}
	As is immediate from Proposition \ref{prop:functoriality}, if the morphism $X
	\to C$ carries an action by a group $H$ which respects the orientation on
	$\Gamma(X_0)$, then $X[n]\to C[n]$ inherits an $H$-action. 
\end{remark}

\subsubsection{Bipartite orientations}

Let $X\to C$ be a projective simple degeneration. The following notion will be
crucial for our construction.

\begin{definition}\label{def:bipartite}
	We say that the dual graph $\Gamma(X_0)$ is \emph{bipartite} if its vertex
	set $V$ can be written as a disjoint union $V = V^+ \union V^-$, such that
	there are no edges between any pair of vertices in $V^+$ or in $V^-$. The
	choice of such a decomposition $V = V^+\union V^-$ induces an orientation of
	$\Gamma(X_0)$ with all arrows pointing from $V^-$ to $V^+$. We shall call
	orientations of this form \emph{bipartite}.
\end{definition}

Equivalently, an orientation is bipartite when every vertex is either a source
or a sink. As is well known, a graph can be given a bipartite orientation if and
only if it has no cycles of odd length, and when this holds, and the graph is
connected, there are exactly two bipartite orientations, obtained from one
another by reversing all arrows. Although the assumption that $\Gamma(X_0)$ is
bipartite is a restriction, one can always produce this situation after a
quadratic base change:

\begin{remark}\label{rem:quadraticext}
Up to a quadratic extension in the base, we can always assume that $\Gamma(X_0)$
admits a bipartite orientation. Indeed, let $ C' \to C $ be a base extension
obtained by extracting a square root of a local parameter at $ 0 \in C $, and
let $ D \subset Y \intsct Y' $ be any component of the double locus in $X_0$.
Then $ X \times_C C' $ acquires a transversal $A_1$-singularity, i.e. a cone
over a conic, along $D \times \{0\}$. Blowing up this  $A_1$-singularity yields
a projective simple degeneration $ X' \to C' $, where the inverse image $
\widetilde{Y}_D $ of $D$ in $X'$ is a $\mathbb{P}^1$-bundle intersecting the
strict transforms of $Y$ and $Y'$ in disjoint sections. Now one can orient
$\Gamma(X'_0)$ by letting all edges point towards the exceptional components.
\end{remark}

In a bipartite orientation there are in particular no directed cycles, so by
Proposition \ref{prop:projectivity-criterion}, all $X[n]\to C[n]$ are
projective.

\subsubsection{Embedding in $\PP^1$-bundles}\label{sec:P1-bundles}

As a convenient tool for describing line bundles on $X[n]$ we next exhibit an
embedding of $X[n]$ into a product of $\PP^1$-bundles over
$X\times_{\AA^1}\AA^{n+1}$. This is a globalized version of the local equations
in Proposition \ref{prop:localeq}.

In Proposition \ref{prop:projectivity-criterion} the desingularization
\begin{equation*}
	X[n] \to X[n-1] \times_{\AA^1}\AA^2
\end{equation*}
was shown to be given globally as a sequence of blow-ups: with $Y$ running
through the components of $X_0$ we blow up along the strict transforms of
$Y\times V(t_{n+1}) \subset X\times_{\AA^1} \AA^{n+1}$ using the orientation of
$\Gamma(X_0)$ to determine the blow-up order. Moreover the penultimate blow-up
already resolves all singularities, so the very last blow-up, corresponding to
sinks in $\Gamma(X_0)$, has a Cartier divisor as centre and thus has now effect.
Thus, in the bipartite situation the above desingularization is a single
blow-up: its centre is the strict transform of $Y_{(0)}\times V(t_{n+1}) \subset
X\times_{\AA^1} \AA^{n+1}$ where $Y_{(0)}$ is the disjoint union of all
components in $X_0$ corresponding to source nodes in $\Gamma(X_0)$.  By e.g.\ a
computation in local coordinates one verifies that this strict transform
coincides with the total transform, so the blow-up centre can be written
\begin{equation*}
	p_{n-1}^{-1}(Y_{(0)})\times V(t_{n+1}) \subset X[n-1]\times_{\AA^1}\times\AA^2
\end{equation*}
where we use coordinates $(t_n,t_{n+1})$ on the last factor $\AA^2$ and
\begin{equation*}
	p_{n-1}\colon X[n-1] \to X
\end{equation*}
is the composition of the resolution $X[n-1]\to X\times_{\AA^1}\AA^n$ with the
projection $X\times_{\AA^1}\AA^n\to X$.

We shall realize the blow-up $X[n]\to X[n-1]\times_{\AA^1}\AA^2$ as the strict
transform of $X[n-1]\times_{\AA^1} \AA^2$ under the blow-up of the product
$X[n-1]\times \AA^2$ along $p_{n-1}^{-1}(Y_{(0)}) \times V(t_{n+1})$. Let
$\sheaf{I}\subset\OO_{X[n-1]\times\AA^2}$ be the ideal sheaf of the latter and
define the rank two vector bundle
\begin{equation*}
	\sheaf{E} = \pr_1^*\OO_{X[n-1]}(-p_{n-1}^* Y_{(0)}) \oplus \pr_2^*\OO_{\AA^2}(-V(t_{n+1}))
\end{equation*}
on $X[n-1]\times \AA^2$. Let $y\in H^0(X, \OO_X(Y_{(0)}))$ be a defining
equation for $Y_{(0)}$. The surjection
\begin{equation*}
	\sheaf{E}
	\xrightarrow{
		\begin{pmatrix}
		p_{n-1}^*y\\ t_{n+1}
		\end{pmatrix}
	}
	\sheaf{I}
\end{equation*}
induces a closed embedding $\PP(\sheaf{I}) \subset \PP(\sheaf{E})$ and the
blow-up of $X[n-1]\times\AA^2$ further embeds into $\PP(\sheaf{I})$. Thus the
strict transform $X[n]$ of $j\colon X[n-1]\times_{\AA^1}\AA^2 \into
X[n-1]\times\AA^2$ inherits a closed embedding into $\PP(j^*\sheaf{E})$.
Moreover, let
\begin{equation*}
	\pi_n\colon X[n-1]\times_{\AA^1}\AA^2
	\to (X\times_{\AA^1}\AA^n)\times_{\AA^1}\AA^2 \iso X\times_{\AA^1}\AA^{n+1}
\end{equation*}
be the canonical projection and define the vector bundle
\begin{equation*}
	\sheaf{F}_n = \pr_1^*\OO_X(-Y_{(0)}) \oplus \pr_2^*\OO_{\AA^{n+1}}(-V(t_{n+1}))
\end{equation*}
on $X\times_{\AA^1}\AA^{n+1}$. Then there is a canonical identification
$j^*\sheaf{E} \iso \pi_n^*\sheaf{F}_n$ and thus we have arrived at a closed
embedding
\begin{equation*}
	X[n] \subset \pi_n^*\PP(\sheaf{F}_n)
\end{equation*}
over $X[n-1]\times_{\AA^1} \AA^2$. We claim that, by iteration, we obtain an
embedding
\begin{equation} \label{eq:P1-bundles}
	X[n] \subset \prod_{i=1}^n P_i
\end{equation}
where the product symbol denotes fibred product over $X\times_{\AA^1}\AA^{n+1}$,
and $P_i$ is the $\PP^1$-bundle over $X\times_{\AA^1} \AA^{n+1}$ obtained by
pulling back $\PP(\sheaf{F}_i)$ over the map $X\times_{\AA^1} \AA^{n+1} \to
X\times_{\AA^1}\times\AA^{i+1}$ that multiplies together the last $n+1-i$
coordinates on $\AA^{n+1}$. For $n=1$ there is nothing to prove, and for $n=2$
there is a commutative diagram
\begin{equation*}
	\begin{tikzcd}
		X[2] &[-6ex] \subset &[-6ex] \pi_2^*P_2 \ar[r]\ar[d]
			& X[1]\times_{\AA^1} \AA^2 \ar[d, "\pi_2"] &[-6ex] \subset &[-6ex] P_1\ar[lld] \\
		& & P_2 \ar[r] & X\times_{\AA^1}\AA^3
	\end{tikzcd}
\end{equation*}
where the square is Cartesian. It follows formally from the diagram that there
is an embedding $X[2] \subset P_1\times P_2$, where the product is over
$X\times_{\AA^1}\AA^3$. The general induction step in proving
\eqref{eq:P1-bundles} is similar.

\subsubsection{The linearization}

Consider first the local situation in Proposition \ref{prop:localeq} and let
$\sheaf{L}_0$ be the ample line bundle $\OO_{(\PP^1)^n}(1,\dots,1)$ pulled back
to $X[n]$. We shall write down a particular linearization of the tensor power
$\sheaf{L}_0^{n+1}$. First let $G[n]'$ be a ``second copy'' of the group $G[n]$,
acting on $X[n]$ via the $(n+1)$'st power map $G[n]'\to G[n]$, sending
$\tilde{\tau}\in G[n]'$ to $\tau=\tilde{\tau}^{n+1}$ in $G[n]$. The induced
$G[n]'$-action on the $i$'th factor $\PP^1$ can be lifted to $\AA^2$ in many
ways; we pick the particular lifting that acts on $(u_i, v_i)\in \AA^2$ by
\begin{equation}\label{eq:local-linearization}
	(u_i, v_i) \overset{\tilde{\tau}}
	\mapsto ((\tilde{\tau}_i)^i u_i, \tilde{\tau}_i^{i-(n+1)} v_i),
\end{equation}
using the coordinates in Remark \ref{rem:tau-coord}. We remark that our
preference for this choice is not obvious at this point, but it will lead to a
well behaved GIT stable locus in Section \ref{sec:GIT-analysis}. The lifted
$G[n]'$-action on $(\AA^2)^n$ gives rise to a $G[n]'$-linearization of
$\sheaf{L}_0$. The kernel of $G[n]'\to G[n]$ acts trivially on
$\sheaf{L}_0^{n+1}$, hence we have defined a $G[n]$-linearization on
$\sheaf{L}=\sheaf{L}_0^{n+1}$.

Now we globalize this construction. For the notation in statement (i) in the
following lemma, we refer to Proposition \ref{prop:functoriality} and Remark
\ref{rem:induced-orientation}.

\begin{lemma}\label{Lemma-modlin}
	Let $X\to C$ be a simple degeneration together with a bipartite orientation
	of the dual graph $\Gamma(X_0)$. Then there exists a particular
	$G[n]$-linearized ample line bundle $\sheaf{L}$ on $X[n]$ such that:
	\begin{itemize}
		\item[(i)]
			(Compatibility with \'etale maps:) Let $f\colon X'\to X$ be an \'etale
			map and give $\Gamma(X'_0)$ the orientation induced by the one on
			$\Gamma(X_0)$. Then $\Gamma(X'_0)$ is bipartite, and if $\sheaf{L}'$
			denotes the corresponding $G[n]$-linearized ample line bundle on
			$X'[n]$, then $\sheaf{L}'$ is isomorphic to the pullback of $\sheaf{L}$
			along $f[n]\colon X'[n] \to X[n]$.
		\item[(ii)]
			(Local description:) In the local situation of Proposition
			\ref{prop:localeq}, the line bundle $\sheaf{L}$ is the $(n+1)$'st power
			of $\OO_{(\PP^1)^n}(1,\dots,1)$ pulled back to $X[n]$, with
			$G[n]$-linearization given by \eqref{eq:local-linearization} as above.
	\end{itemize}
\end{lemma}

\begin{proof}
	We use the notation from Section \ref{sec:P1-bundles}.
	Consider the following diagram:
	\begin{equation*}
		\begin{tikzcd}
			X[n] &[-6ex] \subset &[-6ex] \tprod_{i=1}^n P_i \ar[r, "\pr_i"]\ar[rd]
				& P_i \ar[r]\ar[d] & \PP(\sheaf{F}_i) \ar[d] \\
			&&& X\times_{\AA^1}\AA^{n+1} \ar[r] & X\times_{\AA^1}\AA^{i+1}
		\end{tikzcd}
	\end{equation*}
	The rightmost horizontal arrows are $G[n]$-equivariant when we let $G[n]$ act
	on the objects to the right via the projection
	\begin{equation}\label{eq:G-projection}
		G[n] \to G[i], \quad
		(\tau_1,\dots,\tau_i, \dots, \tau_n) \mapsto (\tau_1,\dots, \tau_i),
	\end{equation}
	where we use the coordinates in Remark \ref{rem:tau-coord}.

	For each $i\le n$, the divisor $V(t_{i+1}) \subset \AA^{i+1}$ is invariant
	under the $G[i]$-action, and hence under the $G[n]$-action via
	\eqref{eq:G-projection}. Hence the locally free sheaf
	\begin{equation*}
		\sheaf{F}_i = \pr_1^*\OO_X(-Y_{(0)}) \oplus \pr_2^*\OO_{\AA^{n+1}}(-V(t_{i+1}))
	\end{equation*}
	on $X\times_{\AA^1}\AA^{n+1}$ has a canonical $G[n]$-linearization (trivial
	on the first summand). The induced $G[n]$-action on $\prod_i P_i$ is
	compatible with the action on $X[n]$.

	Since $\sheaf{F}_i$ itself is $G[n]$-linearized, the $G[n]$-action on
	$\PP(\sheaf{F}_i)$ lifts to the geometric vector bundle $\VV(\sheaf{F}_i)$,
	and hence comes with a canonical linearization with underlying line bundle
	$\OO_{\PP(\sheaf{F}_i)}(1)$. In the local situation of Proposition
	\ref{prop:localeq}, the lifted action can be checked to be given in the
	fibres by
	\begin{equation*}
		(u_i, v_i) \overset{\tau}{\mapsto} (u_i, \tau_i^{-1} v_i).
	\end{equation*}
	Guided by equation \eqref{eq:local-linearization} we thus pick the
	$G[n]'$-action (where again $G[n]' \to G[n]$ is the $(n+1)$'st power map) on
	$\VV(\sheaf{F}_i)$ given by the canonical action via $G[n]$ followed by
	scalar multiplication in the fibres of the vector bundle $\VV(\sheaf{F}_i)$
	by the factor $\tilde{\tau}_i^i$. This induces the required
	$G[n]$-linearization of $\OO(n+1)$ on $\PP(\sheaf{F}_i)$ for each $i$. Pull
	these back to $\prod_{i=1}^n P_i$ and form their tensor product. Restrict to
	$X[n]$ to obtain the required linearized line bundle $\sheaf{L}$. It fulfills
	(ii) by construction.

	For the compatibility with \'etale maps $f\colon X'\to X$ note that the
	union $Y'_{(0)}$ of components in $X'_0$ being source nodes in
	$\Gamma(X_0')$ is precisely the inverse image of the corresponding union
	$Y_{(0)}$ in $X_0$. Thus the vector bundles $\sheaf{F}_i$ over
	$X\times_{\AA^1}\AA^{i+1}$ pull back to the corresponding bundles
	$\sheaf{F}'_i$ over $X'\times_{\AA^1}\AA^{i+1}$ and so the product $\prod_i
	P_i$ of $\PP^1$-bundles over $X\times_{\AA^1}\AA^{n+1}$ pulls back to the
	corresponding product $\prod_i P'_i$ over $X'\times_{\AA^1}\AA^{n+1}$. It is
	thus enough to check that the embeddings of $X[n]'$ and $X[n]$ in their
	respective product bundles are compatible. We leave the details to the
	reader.
\end{proof}

\subsubsection{Hilbert--Mumford invariants}\label{subsubsec-HMinvariants}

We first recall the definition of the Hilbert--Mumford invariants. Let $G$
denote a linearly reductive group over $k$, which acts on a quasi-projective
$k$-scheme $Y$. Assume moreover that we are given an ample $G$-linearized
invertible sheaf $\mathcal{P}$ on $Y$. Let $\lambda \colon \Gm \to G$ be a
one-parameter subgroup (for short $1$-PS) of $G$ and $ y \in Y $ a point. If
the limit $y_0$ of $y$ as $\tau \in \Gm $ tends to zero exists in $Y$, then
$y_0$ is a $\Gm$-fixed point and we define the value
$\mu^{\mathcal{P}}(\lambda,y)$ to be the negative of the $\Gm$-weight on the
fibre $\mathcal{P}(y_0)$. Otherwise, we put $\mu^{\mathcal{P}}(\lambda,y) =
\infty$.

As preparation for the application of GIT in Section \ref{sec:GIT-analysis}, we
shall compute the Hilbert--Mumford invariants
$\mu^{\sheaf{L}}(\lambda_{\mathbf{s}}, P)$ associated to arbitrary one parameter
subgroups
\begin{equation}\label{eq:1PS}
	\lambda_{\mathbf{s}} \colon \Gm \to G[n], \quad
	\tau \mapsto (\tau^{s_1},\tau^{s_2},\dots,\tau^{s_n}),
\end{equation}
where $\mathbf{s} = (s_1,\dots,s_n)$ is an $n$-tuple of integers, $P$ is a
point in $X[n]$ and we use $\tau$-coordinates on $G[n]$ as in Remark
\ref{rem:tau-coord}. Let
\begin{equation*}
	P_0 = \lim_{\tau\to 0} \lambda_{\mathbf{s}}(\tau)\cdot P \in X[n]
\end{equation*}
provided the limit exists.

We shall write $t_i(P)$ for the $i$'th coordinate of the image of $P\in X[n]$ in
$\AA^{n+1}$. We use the notation $Y_I$ and $\Delta^{D,i}_I$ introduced in
Section \ref{sec:fibres}, and to avoid writing out special cases it is
convenient to define
\begin{equation*}
	\Delta^{D, 0}_I = Y'_I,\quad \Delta^{D, \max I}_I = Y_I
\end{equation*}
whenever $[Y'] \xrightarrow{[D]} [Y]$ is an arrow in $\Gamma(X_0)$. For the same
reason we let $s_0 = s_{n+1} = 0$.

\begin{proposition}\label{prop:pointweight}
	Let $P\in X[n]$ and let $P_0$ be its limit under a $1$-PS \eqref{eq:1PS} as
	above, provided it exists. Define
	\begin{equation*}
		I = \{ i \mid t_i(P) = 0 \}
	\end{equation*}
	so that $P \in X[n]_I$.
	\begin{itemize}
		\item[(a)]
			The limit $P_0$ exists if and only if $s_{i-1} \le s_i$ for all
			$i\not\in I$. If this is the case, we have $t_i(P_0)=0$ if and only if
			$i$ is in
			\begin{equation*}
				J = I \union \{ i \mid s_{i-1} < s_i \},
			\end{equation*}
			so $P_0\in X[n]_J$.
		\item[(b)]
			Assume the limit $P_0$ exists and $X[n]_I$ is smooth at $P$,
			so that $P$ is in a unique component $\Delta^{D,i}_I$ of $X[n]_I$. Let
			$j>i$ be the successor to $i$ in $I\union \{n+2\}$. By part (a), we
			have $s_i \le s_{i+1} \le \cdots \le s_{j-1}$.
			\begin{itemize}
				\item[(i)]
					Assume all $s_k \ne 0$ for $i\le k < j$. Then $i\ne 0$ and $i\ne
					\max I$. Define $a$ ($i\le a \le j$) by the property $s_k<0$ if
					and only if $k<a$, for all $i\le k <j$. Then $P_0 \in
					\Delta_J^{D,a'} \intsct \Delta_J^{D,a}$, where $a'<a$ is the
					predecessor to $a$ in $J\union \{0\}$, and
					$\mu^{\sheaf{L}}(\lambda_{\mathbf{s}}, P)$ is the sum over all
					$k=1,2,\dots,n$ of contributions
					\begin{align*}
						-k s_k \quad \text{for $k<a$},\\
						(n+1-k)s_k\quad \text{for $k\ge a$}.
					\end{align*}
				\item[(ii)]
					Assume at least one $s_k=0$ for $i\le k < j$. Define $a$ ($i\le a
					\le j$) and $b$ ($i\le b \le j$) by the property $s_k=0$ if and
					only if $a\le k < b$, for all $i\le k < j$. Then $P_0 \in
					\Delta^{D,a}_J$, $X[n]_J$ is smooth at $P_0$, and
					$\mu^{\sheaf{L}}(\lambda_{\mathbf{s}}, P)$ is the sum over all
					$k=1,2,\dots,n$ of contributions
					\begin{align*}
						-k s_k \quad \text{for $k<a$},\\
						(n+1-k)s_k \quad \text{for $k\ge b$}.
					\end{align*}
			\end{itemize}
		\item[(c)]
			Assume the limit $P_0$ exists and $X[n]_I$ is singular at $P$, so that
			$P \in \Delta^{D,i}_I\intsct \Delta^{D,j}_I$ for a consecutive pair
			$i<j$ in $I\union \{0,n+1\}$. Then $P_0 \in \Delta^{D,i}_J \intsct
			\Delta^{D,j}_J$ and $\mu^{\sheaf{L}}(\lambda_{\mathbf{s}}, P)$ is the
			sum over all $k=1,2,\dots,n$ of contributions
			\begin{align*}
				-k s_k \quad \text{for $k<j$},\\
				(n+1-k)s_k \quad \text{for $k\ge j$}.
			\end{align*}
	\end{itemize}
\end{proposition}

\begin{remark}
	In case (b), the contribution to $\mu^{\sheaf{L}}(\lambda_{\mathbf{s}}, P)$
	for $k$ in the range $i\le k < j$ may be written
	\begin{equation*}
		\left(\tfrac{n+1}{2} - k\right) s_k + \tfrac{n+1}{2} |s_k|
	\end{equation*}
	regardless of the values of $a$ and $b$.
\end{remark}

\begin{proof}
	Since $\pi\colon X[n]\to C[n]$ is proper, existence of the limit for $P$ is
	equivalent to the existence of the limit for $Q=\pi(P)$. The $G[n]$-action on
	$C[n] = C\times_{\AA^1}\AA^{n+1}$ is a pullback from $\AA^{n+1}$, on which
	$\sigma \in G[n]$ acts by
	\begin{equation*}
		(t_1,\dots, t_{n+1}) \mapsto (\sigma_1 t_1, \dots, \sigma_{n+1} t_{n+1}).
	\end{equation*}
	The $1$-PS $\lambda_{\mathbf{s}} \colon \Gm \to G[n]$ is given in
	$\sigma$-coordinates by $\sigma_i = \tau^{s_i - s_{i-1}}$. If $t_i$ is
	nonzero, the limit of $\tau^{s_i - s_{i-1}}t_i$, as $\tau$ approaches zero,
	exists if and only if the exponent $s_i-s_{i-1}$ is nonnegative. More
	precisely, the limit equals $t_i$ if $s_i = s_{i-1}$ and it is $0$ if $s_i >
	s_{i-1}$. This proves (a).
	
	In view of Lemma \ref{Lemma-modlin}, the $\Gm$-weight can be computed in the
	\'etale local coordinates from Section \ref{sec:localcoordinates}. Let
	$i<j$ be consecutive elements in $I\union\{0,n+1\}$. In the \'etale local
	coordinates, as one easily verifies, the component $\Delta^i_I$ of $X[n]_I$
	is given by the vanishing of $t_k$ for $k\in I$, together with
	\begin{align*}
		(u_k:v_k) &= (1:0) & &\text{for $k<i$,} \\
		(u_k:v_k) &= (0:1) & &\text{for $k\ge j$,}
	\end{align*}
	and (consequently) $\Delta^i_I\intsct \Delta^j_I$ is given by
	\begin{align*}
		(u_k:v_k) &= (1:0) & &\text{for $k<j$,} \\
		(u_k:v_k) &= (0:1) & &\text{for $k\ge j$.}
	\end{align*}
	
	Clearly $G[n]$ acts trivially on $(u_k:v_k)=(1:0)$ and $(u_k:v_k)=(0:1)$, so
	to locate the limit point $P_0$ it remains only to work out the action on
	$(u_k : v_k)$ for the remaining range $i\le k < j$. The action by the $1$-PS
	is given by $(u_k:v_k) \mapsto (\tau^{s_k} u_k: v_k)$. In case (b), when $P$
	is on a unique component $\Delta^i_I$, we have $(u_k: v_k) \ne (1:0)$ and
	$(u_k:v_k) \ne (0:1)$ for $k$ in the range $i \le k < j$.

	In case (b.i), we have
	\begin{equation*}
		\underbrace{s_i \le \cdots \le s_{a-1}}_{<0}
		\le \underbrace{s_a \cdots \le s_{j-1}}_{>0}
	\end{equation*}
	and thus the limit point $P_0$ has coordinates
	\begin{align*}
		(u_k:v_k) &= (1:0) & &\text{for $k<a$} \\
		(u_k:v_k) &= (0:1) & &\text{for $k\ge a$}
	\end{align*}
	which shows $P_0\in \Delta^{a'}_J \intsct \Delta^a_J$ as claimed.
	
	In case (b.ii), we have
	\begin{equation*}
		\underbrace{s_i \le \cdots \le s_{a-1}}_{<0} \le
		\underbrace{s_a \le \cdots \le s_{b-1}}_{=0} \le
		\underbrace{s_b \cdots \le s_{j-1}}_{>0}
	\end{equation*}
	and thus the limit point $P_0$ has coordinates
	\begin{align*}
		(u_k:v_k) &= (1:0) & &\text{for $k<a$} \\
		(u_k:v_k) &= (0:1) & &\text{for $k\ge b$}
	\end{align*}
	with the remaining $(u_k:v_k)$, for $a\le k < b$ equal to those of $P$. Thus
	$P_0\in \Delta^a_J$ and it is a smooth point in $X[n]_J$.
	
	In case (c), $P$ has coordinates
	\begin{align*}
		(u_k:v_k) &= (1:0) & &\text{for $k<j$} \\
		(u_k:v_k) &= (0:1) & &\text{for $k\ge j$}
	\end{align*}
	and the $\Gm$-action does not change these, so $P_0$ has the same
	$(u_k:v_k)$-coordinates: thus $P_0 \in \Delta^i_J \intsct \Delta^j_J$.
	
	It remains to write down the weights for the induced $\Gm$-action on
	$\sheaf{L}(P_0)$.  Consider the $1$-PS $\lambda_{\mathbf{s}}^{n+1}\colon \Gm
	\to G[n]$ obtained by composing $\lambda_{\mathbf{s}}$ with the $(n+1)$'st
	power map. By definition of the linearized line bundle
	$\sheaf{L}=\sheaf{L}_0^{n+1}$ in Section \ref{sec:linearization}, the
	$\lambda_{\mathbf{s}}$-weight on $\sheaf{L}(P_0)$ agrees with the
	$\lambda_{\mathbf{s}}^{n+1}$-weight on $\sheaf{L}_0(P_0)$. Since
	$\sheaf{L}_0$ is the tensor product of the pullbacks of the tautological
	bundles $\OO_{\PP^1}(1)$ on each factor in $(\PP^1)^n$, the total
	$\lambda_{\mathbf{s}}^{n+1}$-weight on $\sheaf{L}_0$ is the sum of
	contributions of $\lambda_{\mathbf{s}}^{n+1}$-weights on each factor $\PP^1$.
	On the $k$'th factor, the $\lambda_{\mathbf{s}}^{n+1}$-linearization is
	defined by the lifted action, from $\PP^1$ to $\AA^2$, for which
	$\tilde\tau\in\Gm$ acts by
	\begin{equation*}
		(u_k, v_k) \mapsto ((\tilde\tau)^{k s_k} u_k, (\tilde\tau)^{(k-(n+1))s_k} v_k).
	\end{equation*}
	The $\lambda_{\mathbf{s}}^{n+1}$-fixed point $P_0$ necessarily has
	coordinates $(u_k : v_k)$ of the form $(1:0)$ or $(0:1)$ for all $k$ with
	$s_k\ne 0$. The $\lambda_{\mathbf{s}}^{n+1}$-weight is thus the sum of $k
	s_k$ over all $k$ for which $(u_k:v_k)=(1:0)$ and $(k-(n+1))s_k$ over all $k$
	for which $(u_k:v_k)=(0:1)$. Reversing the signs gives the claimed
	expressions for $\mu^{\sheaf{L}}(\lambda_{\mathbf{s}}, P)$.
\end{proof}


\section{GIT-analysis}\label{sec:GIT-analysis}

The $G[n]$-linearized invertible sheaf $\sheaf{L}$ on $X[n]$ constructed in
Lemma \ref{Lemma-modlin} gives rise to a certain ample linearized invertible
sheaf $\sheaf{M}_{\ell} $ on $ \mathbf{H}^n := \Hilb^n(X[n]/C[n])$ (the integer
$\ell \gg 0$ plays only a formal role). In this section we apply a relative
version of the Hilbert--Mumford criterion to carry out a detailed analysis of
(semi-)stability for points in $ \mathbf{H}^n $, with respect to
$\sheaf{M}_{\ell} $. This leads to our main result in this section, Theorem
\ref{theorem-stablelocus}, which provides a detailed combinatorial description
of the (semi-)stable locus.

\subsection{Relative GIT}\label{subsec-relgit}

We first give a brief summary of how Mumford's Geometric Invariant Theory
\cite{GIT} can be carried out in a relative setting. For further details, we
refer to \cite{GHH-2015}.

Let $S = \Spec A$ be an affine scheme of finite type over $k$, and let $f \colon
Y \to S$ be a projective morphism. Let $G$ be an affine linearly reductive group
over $k$. Assume that $G$ acts on $Y$ and $S$ such that $f$ is equivariant. Let
$\mathcal{P} $ be an ample $G$-linearized invertible sheaf on $Y$. Then one can
define the set of stable points $Y^{s}(\mathcal{P})$ and the set of semi-stable
points $Y^{ss}(\mathcal{P})$ in a similar fashion as in the absolute case. These
sets are open and invariant. For the semi-stable locus, there exists a
universally good quotient  
\begin{equation*}
	\phi \colon Y^{ss}(\mathcal{P}) \to Z.
\end{equation*}
We shall often refer to $Z$ as the \emph{GIT quotient} of $Y$ by $G$. Moreover,
there is an open subscheme $ \widetilde{Z} \subset Z $ with $Y^{s}(\mathcal{P})
= \phi^{-1}(\widetilde{Z})$, such that the restriction 
\begin{equation*}
Y^{s}(\mathcal{P}) \to \widetilde{Z}
\end{equation*}
is a universally geometric quotient. For the applications in this paper, it is
of particular importance to note that $Z$ is relatively projective over the
quotient $S/G = \Spec A^G$.

The main tool we shall use in order to compute the (semi-)stable locus, is a
relative version of the well-known Hilbert--Mumford numerical criterion. This
can be formulated as follows \cite[Cor.~1.1]{GHH-2015} (recall our definition of
the Hilbert--Mumford invariants in \ref{subsubsec-HMinvariants}).

\begin{proposition}\label{prop-RelHilbMum-crit}
	Let $y \in Y$ be a point.
	\begin{enumerate}
		\item
			The point $y$ is stable if and only if
			$\mu^{\mathcal{P}}(\lambda,y) > 0$ for every nontrivial $1$-PS
			$\lambda \colon \Gm \to G$.
		\item
			The point $y$ is semi-stable if and only if
			$\mu^{\mathcal{P}}(\lambda,y) \geq 0$ for every $1$-PS
			$\lambda \colon \Gm \to G$.
	\end{enumerate}
\end{proposition}

\subsection{Notation and setup}

Let $C = \Spec A$ be a smooth, connected affine $k$-curve and let $X \to C$ be a
\emph{projective} simple degeneration. We assume that the dual graph $\Gamma :=
\Gamma(X_0)$ allows a bipartite orientation, and we keep fixed one of the two
possible such orientations throughout this section. 

Let $X[n] \to C[n]$ be the $n$-th expanded degeneration of $X \to C$ (with
respect to the given orientation on $\Gamma$). By Proposition
\ref{prop:projectivity-criterion}, the model $X[n]$ is again projective over
$C[n]$. We denote by $\sheaf{L}$ the $G[n]$-linearized line bundle on $X[n]$
constructed in Lemma \ref{Lemma-modlin}. Finally let $\lambda_{\mathbf{s}}$ be a
one parameter subgroup of $G[n]$ as in \eqref{eq:1PS}.

\subsubsection{The determinant line bundle}\label{subsec-Hilblin}

The relative Hilbert scheme $\mathbf{H}^n := \Hilb^n(X[n]/C[n])$ is again
projective over $C[n]$, and it inherits an action by $G[n]$ such that the
structural map $\mathbf{H}^n \to C[n]$ is equivariant. Let $\mathbf{Z}^n \subset
\mathbf{H}^n \times_{C[n]} X[n]$ be the universal family and denote by $p$ and
$q$ the first and second projections, respectively. Then the line bundle
\begin{equation*}
	\sheaf{M}_{\ell} :=
	\mathrm{det}~p_* \left( q^*\sheaf{L}^{\tensor \ell} \vert_{\mathbf{Z}^n} \right) 
\end{equation*}
is relatively ample when $\ell \gg 0$ \cite[Prop.~2.2.5]{HL-2010}, and it
inherits a $G[n]$-linearization from $\sheaf{L}$ (cf. e.g. the discussion in
\cite[Page 90]{HL-2010}). To simplify notation, we write $\sheaf{M}$ instead of
$\sheaf{M}_{1}$.

\subsubsection{Reduction to smooth subschemes}

Let us fix a $1$-PS $\lambda_{\mathbf{s}}$ and a point $[Z] \in \mathbf{H}^n$.
Assume that the limit of $\lambda_{\mathbf{s}}(\tau) \cdot Z$ as $\tau$ goes to
zero exists in $\mathbf{H}^n $; we denote this limit by $Z_0$. Then $\Gm$ acts
on the fibre of $\sheaf{M}_{\ell}$ at $Z_0$, and we will now investigate this
representation in some detail. 

We decompose the limit as 
\begin{equation*}
	Z_0 = \Union_{P} Z_{0,P}, 
\end{equation*}
with $Z_{0,P}$ a finite subscheme of length $n_P$ supported in $P$. Now
$\OO_{Z_0} \tensor \sheaf{L}$ is trivial as a line bundle on $Z_0$, but its
$\Gm$-action is nontrivial. Writing $\sheaf{L}(P)$ for the fibre of $\sheaf{L}$
at $P$, we have an isomorphism
\begin{equation*}
	H^0(\OO_{Z_0} \tensor \sheaf{L})
	= \bigoplus_P \big(H^0(\OO_{Z_{0,P}})\tensor \sheaf{L}(P)\big)
\end{equation*}
as $\Gm$-representations. Taking determinants, we find 
\begin{align*}
	\wedge^n H^0(\OO_{Z_0} \tensor \sheaf{L})
	&= \Tensor_P \wedge^{n_P}\big(H^0(\OO_{Z_{0,P}})\tensor \sheaf{L}(P)\big) \\
	&= \Tensor_P \wedge^{n_P}\big(H^0(\OO_{Z_{0,P}})\big) \tensor \sheaf{L}(P)^{n_P}\\
	&= \Big(\wedge^nH^0(\OO_{Z_0})\Big) \tensor \Big(\Tensor_P \sheaf{L}(P)^{n_P}\Big).
\end{align*}

\begin{definition}\label{def:boundcomb}
	We define the \emph{bounded weight} $\mu^{\mathscr{M}}_b(\mathbf{s},Z)$,
	resp.~the \emph{combinatorial weight} $\mu^{\mathscr{M}}_c(\mathbf{s},Z)$, to
	be the negative of the $\mathbb{G}_m$-weight on
	$\wedge^nH^0(\mathcal{O}_{Z_0})$, resp.~on $\bigotimes_P
	\mathscr{L}(P)^{n_P}$.
\end{definition}

Having made this definition we can, accordingly, write the negative of the
$\mathbb{G}_m$-weight on $\wedge^n H^0(\mathcal{O}_{Z_0} \otimes \mathscr{L})$
as a sum of $\mu^{\mathscr{M}}_b(\mathbf{s},Z)$ and
$\mu^{\mathscr{M}}_c(\mathbf{s},Z)$. Clearly, if we replace $\mathscr{L}$ by
$\mathscr{L}^{\ell}$ in these expressions we find, for any $ \mathbf{s} $, that
$$ \mu^{\mathscr{M}_{\ell}}_c(\mathbf{s},Z) = \ell \cdot
\mu^{\mathscr{M}}_c(\mathbf{s},Z) $$ and that $$
\mu^{\mathscr{M}_{\ell}}_b(\mathbf{s},Z) = \mu^{\mathscr{M}}_b(\mathbf{s},Z), $$
since the bounded weight only depends on the underlying limit subscheme $Z_0$.

Now if $\ell \gg 0$, we in fact have that 
\begin{equation*}
	\sheaf{M}_{\ell}(Z_0) = \wedge^n H^0(\OO_{Z_0} \tensor \sheaf{L}^{\ell}), 
\end{equation*} 
thus we obtain a sum-decomposition of the Hilbert--Mumford invariant attached to
$\sheaf{M}_{\ell}$, $\mathbf{s}$ and $Z$:
\begin{equation}\label{eq:weightdecomp}
	\mu^{\sheaf{M}_{\ell}}(\mathbf{s},Z)
	= \mu^{\sheaf{M}_{\ell}}_b(\mathbf{s},Z)
	+ \mu^{\sheaf{M}_{\ell}}_c(\mathbf{s},Z). 
\end{equation}
Since the right hand side is defined for all $\ell \in \mathbb{N}$, we formally
use the expression $\mu^{\sheaf{M}_{\ell}}(\mathbf{s},Z)$ to denote the above
sum in all cases.
 
Note that for every $Z$ and $\mathbf{s}$ in the situation above, the value
$\mu^{\sheaf{M}}_c(\mathbf{s},Z)$ only depends on the underlying cycle of $Z$,
and not on its scheme structure. This fact is why we chose the terminology {\em
combinatorial weight}. The terminology {\em bounded weight}, however, is
explained by the following lemma.

\begin{lemma}\label{lemma-nonred1}
	Let $[Z] \in \mathbf{H}^n$ and let $\mathbf{s} \in \mathbb{Z}^n$ be any
	element such that the limit of $\lambda_{\mathbf{s}}(\tau) \cdot Z$, as
	$\tau$ goes to zero, exists. Then there are integers $a_i =
	a_i(Z,\mathbf{s})$ such that 
	\begin{equation*}
	 \mu^{\sheaf{M}}_b(\mathbf{s},Z) = \sum_{i=1}^n a_i s_i, 
	\end{equation*}
	where $\vert a_i \vert \leq 2 n^2$ for every $i$.
\end{lemma}

\begin{proof}
	Let $q \in C[n]$ be the point such that the limit $Z_0$ of $Z$ is contained
	in the fibre $X[n]_q$. Then $q$ is a $\Gm$-fixpoint. Let
	$\widetilde{D}\subset X_0$ be the singular locus and denote by
	$\widetilde{\Delta}\subset X[n]$ the inverse image of $\widetilde{D}\times_C
	C[n]$ under the $G[n]$-equivariant map $X[n] \to X\times_C C[n]$. This map
	restricts to an isomorphism $X[n] \setminus \widetilde{\Delta} \to
	(X\setminus \widetilde{D})\times_C C[n]$, and it follows that the
	$\Gm$-action on each $Z_{0,P}$ is trivial (so the weight is zero) unless
	$Z_{0,P}$ is supported on $\widetilde{\Delta}$.
	
	Now we consider the case where $P$ is a point in $\widetilde{\Delta}$.
	Because $Z_{0,P}$ is a finite local scheme, with $P$ a $\Gm$-fixpoint, we can
	work \'etale locally and use the coordinates from Section
	\ref{sec:localcoordinates}. More precisely, locally at $P$, we can find an
	\'etale chart $W_{j+1}$ with coordinates $t_1, \ldots, t_{n+1}, \tilde{v}_j,
	\tilde{u}_{j+1}$, with relation $t_{j+1} =
	\tilde{v}_j \tilde{u}_{j+1}$, and, depending on the relative dimension of $X \to C$, 
	additional coordinates $\{z_{\alpha}\}_{\alpha}$ (subject to no relations). 
	Here we write, for simplicity, $\tilde{v}_j$
	and $\tilde{u}_{j+1} $ instead of $v_j/u_j$ and $u_{j+1}/v_{j+1}$,
	respectively. Since $P \in \widetilde{\Delta}$ we can assume $t_{j+1}(P) =
	0$, which implies that $\tilde{v}_j \tilde{u}_{j+1} = 0$ at $P$ as well. 
	
	If $\tilde{u}_{j+1} \neq 0$ or $\tilde{v}_{j} \neq 0$ at $P$, then, by the
	fact that $P$ is a $\Gm$-fixpoint, a direct computation using our coordinates
	shows that $\Gm$ acts trivially in an \'etale neighbourhood of $P$ in
	$X[n]_q$, and hence on $Z_{0,P}$. 
	
	If $\tilde{u}_{j+1} = \tilde{v}_{j} = 0$ at $P$, then the coordinate ring of
	$Z_{0,P}$ is spanned by $n_P$ monomials $M_{P,r}$ in the variables $
	\tilde{v}_{j}, \tilde{u}_{j+1}$ and the $z_{\alpha}$-s, with each monomial
	necessarily of degree at most $n_P$. As $\tilde{u}_{j+1}$ and $\tilde{v}_{j}$
	are semi-invariant with weights $s_{j+1}$ and $-s_{j}$, whereas the
	$z_{\alpha}$-s are invariant, it follows that the $\Gm$-weight for each
	monomial $M_{P,r}$ is of the form $- c_{r,j} s_j + c_{r,j+1} s_{j+1}$, where
	$c_{r,j}$, resp.~$c_{r,j+1}$, denotes the multiplicity of $\tilde{v}_j$,
	resp.~$\tilde{u}_{j+1}$, in $M_{P,r}$. In particular, $c_{r,j}$ and
	$c_{r,j+1}$ are bounded by $n_P$. 
	
	Now we sum over all the points $P$ in the support of $Z_0$. Since the
	integers $ n_P $ sum up to $ n $ as $P$ runs over the points in the support
	of $Z_{0}$, we arrive at the asserted expression for the weight on
	$\wedge^nH^0(\OO_{Z_0})$.
\end{proof}

The following lemma states that, under certain conditions, the combinatorial
weight will dominate the bounded weight, provided that we replace $\sheaf{L}$ by
a sufficiently high tensor power. 

\begin{lemma}\label{prop-nonred}
	Let $[Z] \in \mathbf{H}^n$ and let $\ell \gg 2 n^2$ be an integer. 
	\begin{enumerate}
		\item
			Assume, for every $\mathbf{s} \in \ZZ^n$ such that the limit of
			$\lambda_{\mathbf{s}}(\tau) \cdot Z$ as $\tau$ goes to zero exists,
			that there exist integers $b_i = b_i(\mathbf{s},Z)$ such that 
			\begin{equation*}
				\mu^{\sheaf{M}}_c(\mathbf{s},Z) = \sum_{i=1}^n b_i s_i, 
			\end{equation*}
			where $b_i s_i > 0$ if $ s_i \neq 0 $. Then $[Z] \in
			\mathbf{H}^n(\sheaf{M}_{\ell})^{s}$.
		\item
			Let $\mathbf{s} \in \ZZ^n$ be a nonzero tuple such that the limit of
			$\lambda_{\mathbf{s}}(\tau) \cdot Z$ as $\tau$ goes to zero exists.
			Assume there exist integers $b_i = b_i(\mathbf{s},Z)$ such that 
			\begin{equation*}
				\mu^{\sheaf{M}}_c(\mathbf{s},Z) = \sum_{i=1}^n b_i s_i, 
			\end{equation*}
			where $b_i s_i < 0 $ if $ s_i \neq 0$. Then $[Z] \notin
			\mathbf{H}^n(\sheaf{M}_{\ell})^{ss}$.
	\end{enumerate}
\end{lemma}

\begin{proof}
	In both cases, using the decomposition in Equation (\ref{eq:weightdecomp})
	and replacing $\sheaf{M}$ by $\sheaf{M}_{\ell}$, we can write
	\begin{equation*}
		\mu^{\sheaf{M}_{\ell}}(\mathbf{s},Z) = \sum_{i=1}^n (a_i + \ell \cdot b_i) s_i. 
	\end{equation*}
	Assume that $s_i \neq 0$. Then, by assumption, we have $b_i \neq 0$ and by
	Lemma \ref{lemma-nonred1} we know that $\vert a_i \vert \leq 2 n^2$. Since
	$\ell > 2 n^2$, it follows that $a_i + \ell \cdot b_i \neq 0$ as well, with
	the same sign as $b_i$.
	
	In case (1), this means that $\mu^{\sheaf{M}_{\ell}}(\mathbf{s},Z) > 0$ for
	any nontrivial 1-PS, so $Z$ is a stable point by Proposition
	\ref{prop-RelHilbMum-crit}. In case (2), this means that the 1-PS
	corresponding to $\mathbf{s}$ is destabilizing for $Z$.
\end{proof}

Lemma \ref{lemma-nonred1} and Lemma \ref{prop-nonred} will be crucial tools when
we analyse (semi-)stability for the $G[n]$-action on $\mathbf{H}^n$. Equipped
with these results, we will prove that, in order to show that a Hilbert point
$Z$ is either stable or unstable (but not strictly semi-stable), we may treat
$Z$ (as well as its limit $Z_0$) just as a $0$-cycle and forget its finer scheme
structure, provided we replace $\sheaf{L}$ by a sufficiently large tensor power.
What is more, we will also see that there are \emph{no} strictly semi-stable
points.

\subsection{Numerical support and combinatorial weight.}

\subsubsection{Index notation}

To any point $[Z] \in \mathbf{H}^n$ we can associate the subset 
\begin{equation*}
	I_{[Z]} = \{ i \mid t_i(Z) = 0 \} \subset [n+1], 
\end{equation*} 
where the $t_i$-s denote coordinates on $\AA^{n+1}$ as usual. As we have
explained in Section \ref{sec:fibres}, this subset determines completely the
combinatorial structure of the fibre $X[n]_q$ of $X[n]$ in which $Z$ sits as a
subscheme. Indeed, by Proposition \ref{prop:components}, the dual graph of
$X[n]_q$ can be identified with the oriented graph $\Gamma_{I_{[Z]}}$.  

For our purposes, it is useful to represent subsets also in terms of certain
tuples of positive integers. To do this, let us fix an integer $1 \leq r \leq
n+1$. Then any tuple
\begin{equation}\label{equation-distinguishedtuple}
	\mathbf{a} = (a_0,a_1, \ldots, a_r, a_{r+1}) \in \mathbb{Z}^{r+2}
\end{equation} 
such that
\begin{equation*}
	1 = a_0 \leq a_1 < \ldots < a_i < \ldots < a_r \leq a_{r+1} = n + 1 
\end{equation*}
determines the subset 
\begin{equation*}
	I_{\mathbf{a}} := \{a_1, \ldots, a_r \} \subset [n+1]. 
\end{equation*}
The values $a_0$ and $a_{n+1}$ have been added for computational convenience and
play only a formal role.

\subsubsection{Smooth support}\label{subsubsec:combsupp}

Let $[D] \colon [Y] \to [Y']$ be an arrow in the oriented graph $\Gamma$. As we
have explained in Section \ref{sec:fibres}, this arrow gets replaced in the
expanded graph $\Gamma_{I_{\mathbf{a}}}$ by a chain of $r$ arrows. The internal
(``white'') nodes in this chain are denoted $[\Delta_{I_{\mathbf{a}}}^{D,
a_{i}}]$. 

When the set $I_{\mathbf{a}}$ is understood, we shall denote by $\Delta^{a_i}$
the (disjoint) union of the components $\Delta_{I_{\mathbf{a}}}^{D, a_{i}}$ of
$X[n]_{I_{\mathbf{a}}} $, as $[D]$ runs over the arrows in $\Gamma$. In order to
get coherent notation, we also denote by $\Delta^{a_0}$ the union of the
components $Y_{I_{\mathbf{a}}}$, where $[Y]$ runs over vertices in $V^{-}$, and
by $\Delta^{a_r}$ the union of the components $Y'_{I_{\mathbf{a}}}$, where
$[Y']$ runs over the vertices in $V^{+}$.

Consider a point $[Z] \in \mathbf{H}^n$ and assume that
\begin{equation*}
	I_{[Z]} = \{ i \mid t_i(Z) = 0 \} =  I_{\mathbf{a}}. 
\end{equation*}
This means that $Z$ is a subscheme in a general fibre of $X[n]_{I_{\mathbf{a}}}
\to C[n]_{I_{\mathbf{a}}}$. To be precise; by \emph{general} we mean that no
other coordinates $t_j$ are zero. As usual, we decompose $Z$ as a disjoint union
$\union_P Z_P$, where $Z_P$ is supported in $P$ and has length $n_P$. 

\begin{definition}
	We say that $Z$ has \emph{smooth support} if each $P \in \mathrm{Supp}(Z)$
	belongs to a unique component of $ X[n]_{I_{\mathbf{a}}}$. 
\end{definition}

Consequently, when $Z$ has smooth support, there exists for each $P \in
\mathrm{Supp}(Z)$ a unique integer $0 \leq i(P) \leq r$ such that $P \in
\Delta^{a_{i(P)}}$.

\begin{definition}\label{def:numericalsupport}
	If $Z$ has smooth support, we define the \emph{numerical support} of $Z$ to
	be the tuple  
	\begin{equation*}
		\mathbf{v}(Z) = \sum_P n_P \cdot \mathbf{e}_{i(P)} \in \mathbb{Z}^{r+1}, 
	\end{equation*}
	where $\mathbf{e}_{i(P)}$ denotes the $i(P)$-th standard basis vector of $
	\mathbb{Z}^{r+1}$.
\end{definition}

In down to earth terms, the numerical support keeps track of the distribution of
the underlying cycle of $Z$ on the $\Delta^{a_i}$-s, for $0 \leq i \leq r$. 

\subsubsection{Repackaging the numerical support}

In order to  work efficiently with the numerical support, we need to introduce
some more notation. First, for fixed integers $r$ and $n$ with $1 \leq r \leq n
+ 1$, we define the set 
\begin{equation*}
	\mathcal{B} = \{ \mathbf{b} = (b_i) \in \ZZ^{r+2}
	\mid 1 = b_0 \leq \ldots \leq b_i \leq \ldots \leq b_{r+1} = n+1 \}.
\end{equation*}
We also define the set  
\begin{equation*}
	\mathcal{V} = \{ \mathbf{v} = (v_i) \in \left( \ZZ_{\geq0} \right)^{r+1}
	\mid \sum_{i=0}^r v_i = n \}. 
\end{equation*}
Observe that there is an obvious bijection of sets $\mathcal{B} \to
\mathcal{V}$ defined by 
\begin{equation*}
	\mathbf{b} = (b_0, \ldots, b_{r+1}) \mapsto \mathbf{v}_{\mathbf{b}}
	:= (b_1 - b_0, \ldots, b_{r+1} - b_r). 
\end{equation*}

Hence, if $[Z] \in \mathbf{H}^n$ is such that $I_{[Z]}$ has cardinality $r$,
then $I_{[Z]} = I_{\mathbf{a}}$ for a suitable element $\mathbf{a} \in
\mathcal{B}$. If, moreover, $Z$ has smooth support, the numerical support
$\mathbf{v}(Z)$ is an element of $\mathcal{V}$. In this situation, we shall
prove that $Z$ is semi-stable if and only if $\mathbf{v}(Z)$ equals
$\mathbf{v}_{\mathbf{a}}$.

\subsubsection{Combinatorial weights}\label{subsubsec-numprep}

We will next explain how we can use the expressions given in Proposition
\ref{prop:pointweight}, for the $\Gm$-weights for points $P \in X[n]$, to
compute the \emph{combinatorial} $\Gm$-weights of a point $[Z] \in \mathbf{H}^n$
with smooth support. 

We fix an integer $1 \leq r \leq n+1$, and a subset $I_{\mathbf{a}} \subset
[n+1]$ of cardinality $r$. We denote by $\mathbf{e}_i \in \mathbb{Z}^{r+1}$ the
$i$-th standard basis vector. For each $ k \in [n]$ and each $\mathbf{s} \in
\mathbb{Z}^{n}$, we define the value $\omega_k(\mathbf{e}_i,\mathbf{s})$ by the
following recipe:

\begin{equation}\label{def:ik-weight}
	\omega_k(\mathbf{e}_i,\mathbf{s}) = 
	\begin{cases}
		- k \cdot s_k, & 1 \leq k < a_i\\
		(\frac{n+1}{2} - k) \cdot s_k +\frac{n+1}{2} \vert s_k \vert, & a_i \leq k < a_{i+1}\\
		(n + 1 - k) \cdot s_k, & a_{i+1} \leq k \leq n\\
	\end{cases}
\end{equation}
Note that if $P \in X[n]_{I_{\mathbf{a}}}$ is a point which belongs to a unique
$\Delta^{a_{i}}$, then Proposition \ref{prop:pointweight} asserts that
\begin{equation*}
	\mu^{\sheaf{L}}(\lambda_{\mathbf{s}},P)
	= \sum_{k=1}^n \omega_k(\mathbf{e}_i,\mathbf{s}), 
\end{equation*}
assuming the limit $P_0$ of $P$ exists.

We next extend the above construction to define a function
\begin{equation*}
	\omega_k(-,\mathbf{s}) \colon \mathcal{V} \to \ZZ 
\end{equation*}
for each $k \in [n]$ and each $ \mathbf{s} \in \ZZ^n $, by setting
\begin{equation*}
	\omega_k(\mathbf{v},\mathbf{s})
	= \sum_{i=0}^r v_i \cdot \omega_k(\mathbf{e}_i,\mathbf{s}). 
\end{equation*}
Finally, we put
\begin{equation}
	\omega(\mathbf{v},\mathbf{s})
	= \sum_{k=1}^n \omega_k(\mathbf{v},\mathbf{s}). 
\end{equation}

Hence, if $[Z] \in \mathbf{H}^n $ is a point with smooth support, and if
$I_{[Z]} = I_{\mathbf{a}} $, it is immediate from Proposition
\ref{prop:pointweight} that the equality 
\begin{equation*}
	\mu_c^{\sheaf{M}_{\ell}}(\lambda_{\mathbf{s}},[Z]) = \ell \cdot \omega(\mathbf{v}(Z),\mathbf{s}) 
\end{equation*}
holds for all $\ell \geq 1$. In other words, the combinatorial weight of $Z$
only depends on its numerical support $\mathbf{v}(Z)$.

\subsubsection{Numerical computations}

We keep the notation and assumptions from Paragraph \ref{subsubsec-numprep}. In
particular, we have fixed an element $\mathbf{a} = (a_0,\ldots,a_r,a_{r+1}) \in
\mathcal{B}$, corresponding to a subset $I_{\mathbf{a}}$. 

\begin{lemma}\label{lemma:k-weight}
	Let $\mathbf{b} = (b_0, \ldots, b_r, b_{r+1})$ be an arbitrary element of
	$\mathcal{B}$ and let $\mathbf{s} \in \mathbb{Z}^{n}$. Then, for each $j \in
	\{0, \ldots, r \}$ and $a_j \leq k < a_{j+1}$, the following hold:
	\begin{enumerate}
		\item
			If $s_k \geq 0$, then 
			\begin{equation*}
				\omega_k(\mathbf{v}_{\mathbf{b}},\mathbf{s})
				= - \vert s_k \vert \cdot \left( (k+1-b_{j+1}) (n+1) - k \right). 
			\end{equation*}
		\item
		If $s_k \leq 0$, then 
		\begin{equation*}
			\omega_k(\mathbf{v}_{\mathbf{b}},\mathbf{s})
			= \vert s_k \vert \cdot \left( (k+1-b_{j}) (n+1) - k \right). 
		\end{equation*}
	\end{enumerate}
\end{lemma}

\begin{proof}
	For any element $\mathbf{v} = (v_0, \ldots, v_r)$ of $\mathcal{V}$, a direct
	computation using Equation (\ref{def:ik-weight}) shows that
	$\omega_k(\mathbf{v}, \mathbf{s}) $ equals
	\begin{equation*}
		\sum_{i = 0}^{j-1} v_i \cdot (n+1-k) s_k + v_j \cdot
		\bigg(\big(\frac{n+1}{2}-k\big) s_k + \frac{n+1}{2} \vert s_k \vert\bigg)
		- \sum_{i = j+1}^r v_i \cdot k s_k. 
	\end{equation*}
	Substituting $v_i = b_{i+1}-b_i$ for each $i \in \{0, \ldots, r \}$ easily
	yields the expressions in case (1) and (2). 
\end{proof}

The following result is a key ingredient in analysing (semi-)stability for
points $[Z]$ with smooth support. In particular, it implies that
$\omega_k(\mathbf{v}_{\mathbf{a}},\mathbf{s}) \geq 0$ for all $\mathbf{s} \in
\ZZ^n$, with equality if and only if $s_k=0$.

\begin{lemma}\label{lemma-num}
	Let $\mathbf{b} = (b_0,b_1, \ldots, b_r,b_{r+1}) \in \mathcal B$ and assume,
	for all $j$ and for all $k$ with $a_j \leq k < a_{j+1}$, that the
	inequalities
	\begin{enumerate}
		\item $(k+1-b_{j+1})(n+1) - k \leq 0$
		\item $(k+1-b_j)(n+1) - k \geq 0$
	\end{enumerate}
	are satisfied. Then $\mathbf{b}$ is equal to the fixed element $\mathbf{a}$.
	Moreover, if this is the case, all inequalities are strict.
\end{lemma}

\begin{proof}
We first consider the case where $ \mathbf{b} = \mathbf{a} $. Then the
\emph{strict} inequalities
\begin{equation*}
	(k+1-a_{j+1})(n+1) - k < 0 
\end{equation*}
and
\begin{equation*}
	(k+1-a_j)(n+1) - k > 0 
\end{equation*}
are immediate from the choice of $k$.

Now let $\mathbf{b}$ be an element in $\mathcal B$, and assume that (1) and (2)
both hold for all $j$ and all $a_j \leq k < a_{j+1}$. We will show that this
implies $\mathbf{b} = \mathbf{a}$. If we put $k = a_{j+1} - 1$ in (1), we find
that
\begin{equation*}
	(a_{j+1} - b_{j+1})(n+1) \leq a_{j+1}-1 
\end{equation*}
which can be rewritten as
\begin{equation*}
	a_{j+1} \leq b_{j+1} + \frac{b_{j+1}-1}{n}. 
\end{equation*}
But observe that either $b_{j+1} = n+1$ or the inequality
\begin{equation*}
	0 \leq \frac{b_{j+1}-1}{n} < 1 
\end{equation*}
holds. In both cases, we get
\begin{equation*}
	a_{j+1} \leq b_{j+1}. 
\end{equation*}

If we instead put $k = a_j$, then (2) yields
\begin{equation*}
	(a_j +1 - b_j)(n+1) \geq a_j 
\end{equation*}
which can be rewritten as
\begin{equation*}
	a_j \geq b_j + \frac{b_j-1}{n} - 1.
\end{equation*}
But either $b_j = 1$, or
\begin{equation*}
	-1 < \frac{b_j-1}{n} - 1 \leq 0 
\end{equation*}
holds. In both cases, it is true that
\begin{equation*}
	a_j \geq b_j. 
\end{equation*}
It follows that $\mathbf{b} = \mathbf{a}$.
\end{proof}

We shall also need the following lemma, in order to analyse the combinatorial
$\mathbb{G}_m$-weights of points $[Z] \in \mathbf{H}^n$ which do \emph{not} have
smooth support.

\begin{lemma}\label{lemma-weight-sing}
	Let $P \in X[n]_{I_{\mathbf{a}}}$ and assume that $P \in \Delta^{a_j}$ for
	some $j \in \{0, \ldots, r \}$. If $P$ is not a smooth point of
	$X[n]_{I_{\mathbf{a}}}$, the inequality 
	\begin{equation*}
		\mu^{\sheaf{L}}(\lambda_{\mathbf{s}},P)
		\leq \sum_{k=1}^n \omega_k(\mathbf{e}_j,\mathbf{s}) 
	\end{equation*}
	holds for every $\mathbf{s} \in \mathbb{Z}^n$.
\end{lemma}

\begin{proof}
	By Proposition \ref{prop:pointweight}, we can write
	\begin{equation*}
		\mu^{\sheaf{L}}(\lambda_{\mathbf{s}},P)
		= \sum_{k=1}^n \widetilde{\omega}_k(\mathbf{s},P), 
	\end{equation*}
	where $\widetilde{\omega}_k(\mathbf{s},P) =
	\omega_k(\mathbf{e}_j,\mathbf{s})$ unless $a_j \leq k < a_{j+1}$. For $k$ in
	this range, one computes that $\widetilde{\omega}_k(\mathbf{s},P) = (n+1-k)
	s_k$ if $P \in \Delta^{a_{j-1}} \intsct \Delta^{a_j}$, and that
	$\widetilde{\omega}_k(\mathbf{s},P) = -k s_k$ if $P \in \Delta^{a_j} \intsct
	\Delta^{a_{j+1}}$. In both cases, the inequality
	$\widetilde{\omega}_k(\mathbf{s},P) \leq \omega_k(\mathbf{e}_j,\mathbf{s})$
	holds, and the assertion follows.
\end{proof}

\subsection{The semi-stable locus}

We are now ready to present our main result in this section, namely  a complete
description of the (semi-)stable locus in $\mathbf{H}^n$ with respect to the
$G[n]$-linearized sheaf $\sheaf{M}_{\ell}$, for any integer $\ell \gg 2n^2$. 

Let $[Z] \in \mathbf{H}^n$, and assume that the associated subset 
\begin{equation*}
	I_{[Z]} \subset [n+1] 
\end{equation*} 
has cardinality $r$. We denote by $\mathbf{a} \in \mathcal{B} $ (where
$\mathcal{B} $ depends on the values $n$ and $r$) the unique element such that
$I_{[Z]} = I_{\mathbf{a}}$. 

\begin{theorem}\label{theorem-stablelocus}
	Let $\ell \gg 2n^2$. The (semi-)stable locus in $\mathbf{H}^n$ with respect
	to $\sheaf{M}_{\ell}$ can be described as follows:
	\begin{enumerate}
		\item
			If $[Z] \in \mathbf{H}^n $ has smooth support, then $[Z] \in
			\mathbf{H}^n(\sheaf{M}_{\ell})^{ss}$ if and only if 
			\begin{equation*}
				\mathbf{v}(Z) = \mathbf{v}_{\mathbf{a}}. 
			\end{equation*} 
			In this case, it also holds that $[Z] \in
			\mathbf{H}^n(\sheaf{M}_{\ell})^{s}$.
		\item
			If $[Z] \in \mathbf{H}^n$ does not have smooth support, then $[Z]
			\notin \mathbf{H}^n(\sheaf{M}_{\ell})^{ss}$.
	\end{enumerate}
\end{theorem}

\begin{proof}
	We consider first the case where $Z$ has smooth support. If $\mathbf{v}(Z) =
	\mathbf{v}_{\mathbf{a}}$, Lemma \ref{lemma-num} states that
	$\mu_c^{\sheaf{M}}(\lambda_{\mathbf{s}},Z) > 0$ for every nontrivial $1$-PS
	$\lambda_{\mathbf{s}}$ such that the limit of $Z$ exists. This implies, by
	Lemma \ref{prop-nonred}, that the same statement holds for $
	\mu^{\sheaf{M}_{\ell}}(\lambda_{\mathbf{s}},Z) $. Thus $[Z] \in
	\mathbf{H}^n(\sheaf{M}_{\ell})^s $ by Proposition \ref{prop-RelHilbMum-crit}. 
	
	Assume instead that $\mathbf{v}(Z) = \mathbf{v}_{\mathbf{b}}$ for some
	element $\mathbf{b} \in \mathcal{B} $ where $ \mathbf{b} \neq \mathbf{a}$. In
	this case we will produce an explicit $1$-PS which is destabilizing for the
	Hilbert point $Z$. 
	
	Assume first that $a_{j+1} > b_{j+1}$, and put $\kappa = a_{j+1} - 1$. Then 
	\begin{equation*}
		(\kappa + 1 - b_{j+1})(n+1) - \kappa
		= (a_{j+1} - b_{j+1})(n+1) - (a_{j+1}-1) > 0. 
	\end{equation*}
	For $d \gg 0$, we define $\mathbf{s} = \mathbf{s}(d) \in \mathbb{Z}^n$ as
	follows. We put $s_i = 0 $, unless $a_j \leq i < a_{j+1}$. We moreover put
	$s_{a_{j+1} - 1} = d $, and, unless $a_{j+1} - 1 = a_j$, we put $s_{a_j} =
	0$. Then we define, inductively, $s_k = s_{k-1} + 1$ for
	$a_j < k < a_{j+1} - 1$. Now we find that the expression
	$\sum_{k=a_j}^{a_{j+1}-2} \omega_k(\mathbf{v}(Z),\mathbf{s})$ is bounded,
	independently of $d$. On the other hand,
	\begin{equation*}
		\omega_{a_{j+1} - 1}(\mathbf{v}(Z),\mathbf{s})
		=  - ((a_{j+1} - b_{j+1})(n+1) - (a_{j+1}-1)) \cdot d < 0. 
	\end{equation*}
	Hence, choosing $d$ sufficiently large yields the desired $1$-PS.
	
	Assume instead that $b_j > a_j$, and set $\kappa = a_j$. Then
	\begin{equation*}
		(\kappa +1 - b_j)(n+1) - \kappa = (a_j + 1 - b_j)(n+1) - a_j < 0. 
	\end{equation*}
	For $d \ll 0$, we define $ \mathbf{s} = \mathbf{s}(d) \in \mathbb{Z}^n$ as
	follows. Put $s_{a_j} = d \ll 0$. Unless $a_{j+1} - 1 = a_j$, we put
	$s_{a_{j+1}-m} = - m$ whenever $1 \leq m \leq a_{j+1} - (a_j + 1)$. Set all
	remaining $s_i = 0$. A similar argument as in the previous case shows that
	this yields a destabilizing $1$-PS for $Z$.
	
	It remains to consider the case where $Z$ does not have smooth support. As
	usual, let $Z = \union_P Z_P$ be the decomposition of $Z$ into local
	subschemes of length $n_P$. We construct two distinct vectors $\mathbf{v}'$
	and $ \mathbf{v}''$ in $\mathcal{V}$ as follows. 
	
	If $P$ belongs to a unique component $\Delta^{a_j}$, we set $\mathbf{v}_P' =
	\mathbf{v}_P'' = n_P \cdot \mathbf{e}_j$. Let $j_{\mathrm{min}}$ be the
	smallest index in $\{0, \ldots, r-1\}$ such that there is at least one point
	in the support of $Z$ belonging to the intersection of
	$\Delta^{a_{j_{\mathrm{min}}}}$ and $\Delta^{a_{j_{\mathrm{min}}+1}}$. For
	each such point $P$, we set $\mathbf{v}_P' = n_P \cdot
	\mathbf{e}_{j_{\mathrm{min}}}$ and $\mathbf{v}_P'' = n_P \cdot
	\mathbf{e}_{j_{\mathrm{min}}+1}$. Finally, if $P$ is a point in the
	intersection of two components $\Delta^{a_j}$ and $\Delta^{a_{j+1}}$ where $j
	> j_{\mathrm{min}}$, we set $\mathbf{v}_P' = \mathbf{v}_P'' = n_P \cdot
	\mathbf{e}_j$.
	
	We now define $\mathbf{v}' := \sum_P \mathbf{v}_P'$ and $\mathbf{v}'' :=
	\sum_P \mathbf{v}_P''$, where the sum runs over all points in the support of
	$Z$. By Lemma \ref{lemma-weight-sing}, both the inequalities
	$\omega(\mathbf{v}(Z),\mathbf{s}) \leq \omega(\mathbf{v}',\mathbf{s})$ and
	$\omega(\mathbf{v}(Z),\mathbf{s}) \leq \omega(\mathbf{v}'',\mathbf{s})$ hold.
	Since $\mathbf{v}' \neq \mathbf{v}''$, at least one of them is different from
	$\mathbf{v}_{\mathbf{a}}$. Hence we can construct a $1$-PS such that $Z$ has
	a limit $Z_0$ in $X[n]$, and which is destabilizing for $Z$, in the same
	fashion as above.
\end{proof}

\subsection{Necessity of bipartite assumption}

We conclude this section by exhibiting an example which shows that the
bipartite condition is in fact crucial. When $\Gamma(X_0)$ has no directed
cycles, but is not necessarily bipartite, the construction of $X[n]$ in
Proposition \ref{prop:projectivity-criterion} by blowing up (invariant) Weil divisors,
immediately leads to (essentially canonical) linearized ample line bundles on
$X[n]$. The ample line bundle we have constructed in the bipartite case is
indeed of this form, but the $G[n]$-action on it  has been modified. This
modified linearization only works in the bipartite situation. The following
example shows that our set-up cannot be extended, at least not simply through a
clever choice of linearization, beyond the bipartite situation.

\begin{figure}
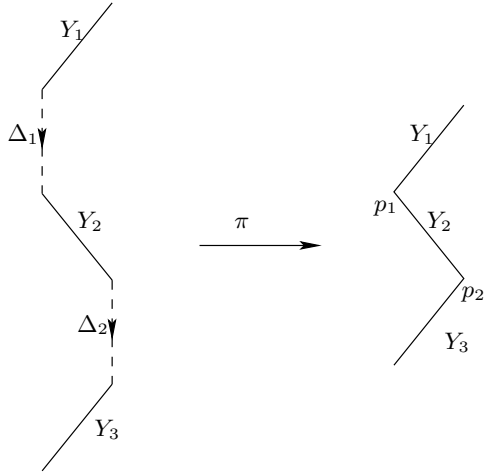

	\centering
	\input nonbipartite.pspdftex
	\caption{$X[1]$ for a non bipartite orientation}\label{fig:nonbipartite}
\end{figure}

\begin{example}\label{ex:bipartite}
	Let $X\to C$ be a curve degeneration with dual graph $\Gamma(X_0)$ of the
	form
	\begin{equation*}
		\bullet \to \bullet \to \bullet
	\end{equation*}
	and choose the (non bipartite) orientation shown. Consider the canonical map
	$\pi\colon X[1]\to X$. We claim: there is no linearization on $X[1]$ such
	that
	\begin{itemize}
		\item[(i)]
			the semi-stable locus $X[1]^{\mathrm{ss}}$ is contained in the smooth
			locus $X[1]^{\mathrm{sm}}$ over $C[1]$, and
		\item[(ii)]
			the image $\pi(X[1]^{\mathrm{ss}})\subset X$ contains the singular
			points of $X$.
	\end{itemize}
	Clearly, the latter condition is  necessary if we also want to capture cycles
	supported on the singular locus of $X_0$. To see this, consider Figure
	\ref{fig:nonbipartite}, showing the degenerate fibre $X[1]_0$ with its
	``old'' components $Y_1$, $Y_2$, $Y_3$, and the ``new'' components $\Delta_1$
	and $\Delta_2$, together with the canonical map to $X_0$. The group
	$G[1]=\Gm$ acts on $\Delta_i$ as indicated by the arrow, whereas $Y_i$ are
	pointwise fixed. For each of the singular points $p_i$ in $X_0$, we have
	\begin{equation*}
		\pi^{-1}(p_i) \intsct X[1]^{\mathrm{sm}} = \Delta_i^\circ
	\end{equation*}
	(where $\Delta_i^\circ$ denotes the interior of $\Delta_i$ in $X[1]_0$). So
	for condition (ii) to hold, the orbits $\Delta_i^\circ$ must be semi-stable.
	Now the $\Gm$-weight on any linearized line bundle is constant along the
	pointwise fixed component $Y_2$, and it cannot be zero, since then
	$Y_2\intsct \Delta_i$ would be semi-stable, violating (i). By the
	Hilbert--Mumford criterion, $\Delta_1^\circ$ is semi-stable only if that
	weight is nonpositive, and $\Delta_2^\circ$ is semi-stable only if that
	weight is nonnegative. This is a contradition.
\end{example}


\section{The quotients}\label{sec:quotients}

In this section, we introduce the stack quotient $\mathcal{I}^n_{X/C}$ and the
GIT quotient $I^n_{X/C}$ of $\mathbf{H}^n(\sheaf{M}_{\ell})^{s}$ by $G[n]$,
where $\ell \gg 0$. We show in Theorem \ref{prop-GITquotient} that
$\mathcal{I}^n_{X/C}$ is proper over $C$, with coarse moduli space $I^n_{X/C}$
(which is projective over $C$). We moreover demonstrate in Theorem
\ref{prop:comparison} that $\mathcal{I}^n_{X/C}$ is isomorphic, as a DM stack
over $C$, to the stack $\mathcal{I}^{P}_{\mathfrak{X}/\mathfrak{C}}$ introduced
by Li and Wu (cf.~e.g.~\cite{LW-2011}), when $P$ is the constant Hilbert
polynomial $n$.

\subsection{Stack quotient and GIT quotient}\label{subsec-twoquotients}

Let $X \to C$ denote a projective simple degeneration, where $C = \Spec A$ is a
smooth affine curve over $k$. We assume that $\Gamma(X_0)$ allows a bipartite
orientation, and we fix one of the two possible such orientations. For any
integer $n > 0$, the expansion 
\begin{equation*}
	X[n] \to C[n] 
\end{equation*}
induces a $G[n]$-equivariant morphism
\begin{equation*}
	\mathbf{H}^n = \Hilb^n(X[n]/C[n]) \to C[n]. 
\end{equation*}
For any integer $\ell \gg 0$, we defined in \ref{subsec-Hilblin} a
$G[n]$-linearized ample line bundle $\sheaf{M}_{\ell}$ on $ \mathbf{H}^n$.

Theorem \ref{theorem-stablelocus} provides, when $\ell \gg 2 n^2$, an explicit
description of the subset 
\begin{equation*}
	\mathbf{H}^n(\sheaf{M}_{\ell})^{s}= \mathbf{H}^n(\sheaf{M}_{\ell})^{ss} \subset \mathbf{H}^n 
\end{equation*}
of (semi-)stable points. As the (semi-)stable locus is independent of the choice
of $\ell$, we will in the sequel denote this set simply by
$\mathbf{H}^n_{\mathrm{GIT}}$.

\begin{definition}\label{def-quot}
	We define the following two quotients:
	\begin{enumerate}
		\item
			The \emph{GIT quotient}
			\begin{equation*}
				I^n_{X/C} = \mathbf{H}^n_{\mathrm{GIT}}/G[n]. 
			\end{equation*} 
		\item
			The \emph{stack quotient} 
			\begin{equation*}
				 \mathcal{I}^n_{X/C} = [\mathbf{H}^n_{\mathrm{GIT}}/G[n]]. 
			\end{equation*} 
	\end{enumerate}
\end{definition}

\begin{theorem}\label{prop-GITquotient}
	The GIT quotient $I^n_{X/C}$ is projective over $C$. The stack
	$\mathcal{I}^n_{X/C}$ is a Deligne-Mumford stack, proper and of finite type
	over $C$, having $I^n_{X/C}$ as coarse moduli space.
\end{theorem}

\begin{proof}
	Since $\sheaf{M}_{\ell}$ is ample (by our assumption $\ell \gg 2n^2 $),
	\cite[Prop.~2.6]{GHH-2015} asserts that $I^n_{X/C}$ is relatively projective
	over the quotient
	\begin{equation*}
		C[n]/G[n] = \Spec(A[n]^{G[n]}), 
	\end{equation*}
	where 
	\begin{equation*}
		A[n] = A \tensor_{k[t]} k[t_1, \ldots, t_{n+1}]. 
	\end{equation*} 
	It is straightforward to check that $A[n]^{G[n]} = A$.
	
	All stabilizers for the action of $G[n]$ on $\mathbf{H}^n_{\mathrm{GIT}}$ are
	finite and reduced, hence, by \cite[(7.17)]{vistoli-1989},
	$\mathcal{I}^n_{X/C}$ is a Deligne-Mumford stack. It is of finite type over
	$C$, as this holds for $\mathbf{H}^n_{\mathrm{GIT}}$.
	
	By \cite[Thm.~2.5]{GHH-2015}, the quotient
	\begin{equation*}
		\mathbf{H}^n_{\mathrm{GIT}} \to I^n_{X/C} 
	\end{equation*}
	is universally a geometric quotient. Therefore, \cite[(2.11)]{vistoli-1989}
	asserts that $I^n_{X/C}$ is a coarse moduli space for $\mathcal{I}^n_{X/C}$.
	In particular, this means that there is a proper morphism
	\begin{equation*}
		\mathcal{I}^n_{X/C} \to I^n_{X/C}. 
	\end{equation*}  
	Since $I^n_{X/C} \to C$ is projective, this implies that
	$\mathcal{I}^n_{X/C}$ is proper over $C$.
\end{proof}

We remark that it follows from Proposition \ref{prop:inversion} that these
quotients do not depend on the choice of bipartite orientation of $\Gamma(X_0)$.
It is moreover clear from the construction that both quotients $I^n_{X/C}$ and
$\mathcal{I}^n_{X/C}$ are isomorphic, over $C^* = C \setminus \{0\}$, to the
family $\Hilb^n(X^*/C^*) \to C^*$. 

\begin{remark}
	If a group $H$ acts equivariantly on $X \to C$, and respecting the
	orientation on $\Gamma(X_0)$, one can show that there is an induced action on
	$I^n_{X/C} \to C$. This holds in particular in the situation described in
	Remark \ref{rem:equivariance}, meaning that the Galois group $\mathbb{Z}/2$
	of the base extension $C'/C$ acts naturally on $I^n_{X'/C'} \to C'$.
\end{remark}

\subsection{Comparison with Li--Wu}\label{sec:comparison}

We would now like to explain the relation between our construction and the
results of Li and Wu. An important ingredient in their work is the so-called
{\em stack of expanded degenerations} $\mathfrak{X}/\mathfrak{C}$. We will only
explain the properties of this stack that are needed for our results in this
section, for further details, we refer to \cite[Ch.~2]{li-2013}.

\subsubsection{Standard embeddings}

First we recall some useful notation and facts, following \cite[Ch.~2]{li-2013}.
For any subset $I \subset [n+1]$, we let $I^{\circ}$ denote its complement in
$[n+1]$. If $\vert I \vert = m + 1$, 
\begin{equation*}
	\iota_I \colon [m+1] \to I \subset [n+1] 
\end{equation*} 
denotes the unique order-preserving map.

We set 
\begin{equation*}
	\AA^{n+1}_{U(I)}
	= \{(t) \in \AA^{n+1} \mid t_i \neq 0, i \in I^{\circ} \}. 
\end{equation*} 
Then there is a canonical isomorphism
\begin{equation*}
	\tilde{\tau}_I \colon \AA^{m+1} \times G[n-m] \to \AA^{n+1}_{U(I)}, 
\end{equation*}
defined by $(t_1', \ldots, t_{m+1}'; \sigma_1, \ldots, \sigma_{n-m}) \mapsto
(t_1, \ldots, t_{n+1})$, where $t_k = t_l'$ if $k = \iota_I(l)$ and $t_k =
\sigma_l$ if $k = \iota_{I^{\circ}}(l)$. Restricting $\tilde{\tau}_I$ to the
identity element of $G[n-m]$ gives what Li calls the \emph{standard embedding}
\begin{equation*}
	\tau_I \colon \AA^{m+1} \to \AA^{n+1}. 
\end{equation*}

For each $n$, let $p_n \colon X[n] \to X$ be the canonical $G[n]$-equivariant
morphism. If $\vert I \vert = m+1$, then $\tau_I$ induces an isomorphism 
\begin{equation*}
	(\tau_I^* X[n], \tau_I^* p_n) \iso (X[m],p_m). 
\end{equation*}
over $C[m]$ \cite[2.14 + 2.15]{li-2013}. (We already encountered a special case
of this in the proof of Proposition \ref{prop:scheme-criterion}.)

\subsubsection{The stack of expanded degenerations}\label{subsubsec-sectiondescription}

Returning to the stack of expanded degenerations, one can give the following
useful description of the objects of this stack. 

Let $T$ be a $C$-scheme. An object $(W,p)$ of $\mathfrak{X}(T)$, also called an
\emph{expanded degeneration} of $X/C$, is a family sitting in a commutative
diagram \cite[Def.~2.21, Prop.~2.22]{li-2013}
\begin{equation*}
	\begin{tikzcd}
		W \ar[r, "p"] \ar[d] & X \ar[d] \\
		T \ar[r] & C
	\end{tikzcd}
\end{equation*}
where $W/T$ is allowed to have \emph{expansions} of $X_0$ \cite[2.2]{li-2013} as
fibres, in addition to the original fibres of $X$. 

More precisely, an \emph{effective family} in $\mathfrak{X}(T)$ is simply the
pullback $\xi^* X[m]$ through a $C$-morphism $\xi \colon T \to C[m]$, for some
$m$, with projection induced by $X[m] \to X$. Two effective families are
\emph{effectively equivalent} if there are standard embeddings $\tau_i \colon
C[m_i] \to C[m]$, $i \in \{1,2\}$, and a $T$-valued point $\sigma \colon T \to
G[m]$, such that 
\begin{equation*}
	\tau_1 \circ \xi_1 = (\tau_2 \circ \xi_2)^{\sigma}. 
\end{equation*}

In general, an expanded degeneration in $\mathfrak{X}(T)$ is a family $W \to T$
where $T$ allows an \'etale cover $\union T_i \to T$ such that $W \times_T T_i$
is effective, and such that the canonical isomorphism over $ T_i \times_T T_j$
is induced by an effective equivalence. Finally, an arrow of two expanded
degenerations $(W,p)$ and $(W',p')$ over $T$ is a $T$-isomorphism $W \to W'$
which is locally an effective equivalence.

\begin{remark}\label{remark-pointequiv}
	Two objects $\xi_1$ and $\xi_2 $ in $\mathfrak{X}(k)$ are equivalent if they
	can be embedded as fibres in the same expanded degeneration $X[n]$, for
	sufficiently large $n$, such that the fibre $\xi_1$ can be `translated' to
	the fibre $\xi_2 $ under the $G[n]$-action. In particular, under this
	equivalence, any object $ \xi $ of $ \mathfrak{X}(k)$ can be represented by a
	fibre $ X[m]_0$, where $ 0 \in C[m]$ denotes the origin, for a suitable $m$.
\end{remark}

\subsubsection{The Li--Wu stack}

Li and Wu have defined a stack $\mathcal{I}^P_{\mathfrak{X}/\mathfrak{C}}$
parame\-tri\-zing \emph{stable} ideal sheaves with fixed Hilbert polynomial $P$,
which we will explain next. To do this, let $J_Z$ be an ideal sheaf on $X[m]_0$,
for some $m \geq 0$. Li and Wu call $J_Z$ \emph{admissible}
\cite[Def.~3.52]{li-2013} if, for every component $D$ of the double locus, the
natural homomorphism
\begin{equation*}
	J_Z \tensor \OO_D \to \OO_D 
\end{equation*}
is injective. Then $J_Z$ is \emph{stable} if it is admissible and if
$\mathrm{Aut}_{\mathfrak{X}}(J_Z)$, the subgroup of elements $\sigma \in G[m]$
such that $\sigma^* J_Z = J_Z$, is finite. In this paper, we shall often call
such ideal sheaves \emph{Li--Wu stable}, in order to separate this notion of
stability from GIT stability. 

Now, for a $C$-scheme $T$, $\mathcal{I}^P_{\mathfrak{X}/\mathfrak{C}}(T)$
consists of all triples $(J_Z, W, p)$, where $(W, p) \in \mathfrak{X}(T)$, and
$J_Z$ is a $T$-flat family of stable ideal sheaves on $W$ with Hilbert
polynomial $P$. Moreover, every morphism $T' \to T $ induces a pullback map
$\mathcal{I}^P_{\mathfrak{X}/\mathfrak{C}}(T) \to
\mathcal{I}^P_{\mathfrak{X}/\mathfrak{C}}(T')$.

We shall refer to $\mathcal{I}^P_{\mathfrak{X}/\mathfrak{C}}$ as the
\emph{Li--Wu stack}. The following fundamental result has been proved by Li and
Wu (cf.~\cite[Thm.~4.14]{LW-2011} and \cite[Thm.~3.54]{li-2013}).

\begin{theorem}\label{theo:properness}
	$\mathcal{I}^P_{\mathfrak{X}/\mathfrak{C}}$ is a Deligne-Mumford stack,
	separated, proper and of finite type over $C$. 
\end{theorem}

We remark that \cite[Thm.~3.54]{li-2013} is formulated under the assumption that
$X_0 = Y \union Y'$ with $Y$, $Y'$ and $Y \intsct Y'$ smooth and irreducible,
whereas \cite[Thm.~4.14]{LW-2011} is formulated for a general simple
degeneration. 

\subsubsection{Li--Wu stability}

For the remainder of this section, we shall only consider the case where $P$ is
constant, in which case Li--Wu stability can be formulated in a simple way. In
the statement, we shall use the following notation. For any $m \in \mathbb{N}$,
and with $I = [m+1]$, we denote by $\Delta^i $ the (disjoint) union of the
components $\Delta^{D,i}_I $ of $X[m]_0$, where $D$ runs over the edges in the
oriented graph $\Gamma(X_0)$.  

\begin{lemma}\label{LiWu-stab-pt}
	Let $Z \subset X[m]_0$ be a subscheme of finite length. Then $Z$ is Li--Wu
	stable if and only if the following properties hold:
	\begin{enumerate}
		\item
			$Z$ is supported on the smooth locus of $X[m]_0$. 
		\item
			$Z$ has non-empty intersection with $\Delta^i$, for all $i \in [m]$.
	\end{enumerate}
\end{lemma}

\begin{proof}
	A straightforward computation shows that $J_Z$ is admissible if and only if
	$Z$ is supported on the smooth locus of $X[m]_0$. For (2), note that the
	$i$-th factor of $G[m]$ acts on $\Delta^i$ by multiplication in the fibres of
	the ruling. This means that the automorphism group is finite if and only if
	$Z$ intersects every $\Delta^i$ nontrivially. 
\end{proof}

Note the similarity with the description of GIT stable subschemes given in
Theorem \ref{theorem-stablelocus}. We shall next compare the locus
$\mathbf{H}^n_{\mathrm{LW}}$ of Li--Wu stable points in $\Hilb^n(X[n]/C[n])$
with the GIT stable locus $\mathbf{H}^n_{\mathrm{GIT}}$. By
\cite[Lem.~4.3]{LW-2011}, $\mathbf{H}^n_{\mathrm{LW}}$ is an open subset, and it
is clearly invariant. The same properties hold for
$\mathbf{H}^n_{\mathrm{GIT}}$.

\begin{lemma}\label{lemma-stabilitycomparison}
	There is a $G[n]$-equivariant open immersion
	\begin{equation*}
		\mathbf{H}^n_{\mathrm{GIT}} \subset \mathbf{H}^n_{\mathrm{LW}} 
	\end{equation*}
	as subschemes in $\Hilb^n(X[n]/C[n])$.
\end{lemma}

\begin{proof}
	As $\mathbf{H}^n_{\mathrm{GIT}}$ and $\mathbf{H}^n_{\mathrm{LW}}$ are both
	open and invariant, we only need to show that any GIT stable subscheme in a
	closed fibre of $X[n] \to C[n]$ is Li--Wu stable. This is clear from Lemma
	\ref{LiWu-stab-pt} and Theorem \ref{theorem-stablelocus}.
\end{proof}

This inclusion is strict in general; by Theorem \ref{theorem-stablelocus} (1), a
Li--Wu stable subscheme $Z$ will fail to be GIT stable if the numerical support
$\mathbf{v}(Z)$ does not equal $\mathbf{v}_{\mathbf{a}}$, where $I_{[Z]} =
I_{\mathbf{a}}$.

\subsubsection{The canonical comparison morphism}

There is an obvious morphism from our quotient $\mathcal{I}^n_{X/C}$ to the
Li--Wu stack $\mathcal{I}^n_{\mathfrak{X}/\mathfrak{C}}$. Indeed, the
restriction to $\mathbf{H}^n_{\mathrm{LW}}$ of the universal family of the
Hilbert scheme corresponds to a $G[n]$-equivariant, surjective and smooth
morphism
\begin{equation*}
	\psi \colon \mathbf{H}^n_{\mathrm{LW}}
	\to \mathcal{I}^n_{\mathfrak{X}/\mathfrak{C}}. 
\end{equation*}
Restriction to the open subscheme $H^n_{\mathrm{GIT}}$ gives 
\begin{equation*}
	\phi \colon \mathbf{H}^n_{\mathrm{GIT}}
	\to \mathcal{I}^n_{\mathfrak{X}/\mathfrak{C}}, 
\end{equation*}
which is again equivariant and smooth. Hence $\phi$ factors through the
quotient $\mathcal{I}^n_{X/C}$, giving a smooth morphism
\begin{equation}\label{morphism-main}
	f \colon \mathcal{I}^n_{X/C} \to \mathcal{I}^n_{\mathfrak{X}/\mathfrak{C}}. 
\end{equation}

\subsubsection{A criterion for isomorphism}

If $\mathfrak{Z}$ is an algebraic stack over $k$, we denote by $\vert
\mathfrak{Z}(k) \vert$ the set of equivalence classes of objects in
$\mathfrak{Z}(k)$.

\begin{lemma}\label{lemma-pointsstab}
	The following properties hold for $f$:
	\begin{enumerate}
		\item
			$\vert f \vert \colon \vert \mathcal{I}^n_{X/C}(k) \vert \to \vert
			\mathcal{I}^n_{\mathfrak{X}/\mathfrak{C}}(k) \vert$ is a bijection.
		\item
			For every object $\xi$ in $\mathcal{I}^n_{X/C}(k)$, $f$ induces an
			isomorphism
			\begin{equation*}
				\mathrm{Aut}(\xi) \to \mathrm{Aut}(f(\xi)) 
			\end{equation*}
			of automorphism groups.
	\end{enumerate}
\end{lemma}

\begin{proof}
	By Remark \ref{remark-pointequiv}, any point $\xi' $ in $\vert
	\mathcal{I}^n_{\mathfrak{X}/\mathfrak{C}}(k) \vert$ can be represented by a
	Li--Wu stable subscheme $Z \subset X[m]_0$ of length $n$, for some $m \leq
	n$. For any subset $I \subset [n+1]$ with $\vert I \vert = m+1$, we can,
	using the standard embedding $\tau_I$, view$ X[m]_0$ as the fibre $(\tau_I^*
	X[n])_0$ of $ X[n]$, where $0 \in C[m]$. In the notation of
	\ref{subsubsec:combsupp}, we then have
	\begin{equation*}
		I = I_{[Z]} = \{ i \mid t_i(Z) = 0 \}. 
	\end{equation*}
	
	On the other hand, as an element of $\mathcal{V} \subset \mathbb{Z}^{m+2}$,
	the numerical support $\mathbf{v}(Z)$ of $Z$ is independent of the choice of
	$I$. Hence, by Theorem \ref{theorem-stablelocus}, there is a \emph{unique}
	$I$ for which $Z$ is also GIT-stable, namely the subset $I_{\mathbf{a}}$
	determined by the preimage $\mathbf{a}$ of $\mathbf{v}(Z)$ in the bijection
	$\mathcal{B} \to \mathcal{V}$. Thus, the $G[n]$-orbit of $Z$ in
	$\mathbf{H}^n_{\mathrm{GIT}}$ is the unique point $ \xi \in \vert
	\mathcal{I}^n_{X/C}(k) \vert $ such that $f(\xi) = \xi'$, which proves (1).
	Clearly, the automorphism groups of $\xi$ and its image $f(\xi)$ coincide as
	subgroups of $G[m]$ in the above construction, which shows (2). 
\end{proof}

In the proof of Theorem \ref{prop:comparison} below, we shall use the following 
standard technical result on stacks, whose proof we omit:

\begin{lemma}\label{lemma-stackstech}
	Let $\mathfrak{X}$ and $\mathfrak{Y}$ be Deligne-Mumford stacks of finite
	type over an algebraically closed field $k$, and let $$ f \colon \mathfrak{X}
	\to \mathfrak{Y} $$ be a representable \'etale morphism of finite type.
	Assume
	\begin{enumerate}
		\item
			$\vert f \vert \colon \vert \mathfrak{X}(k) \vert \to \vert
			\mathfrak{Y}(k) \vert$ is bijective.
		\item
			For every $x \in \mathfrak{X}(k)$, $f$ induces an isomorphism
			\begin{equation*}
			Aut_{\mathfrak{X}}(x) \to Aut_{\mathfrak{Y}}(f(x)).
			\end{equation*}
	\end{enumerate}
	
	Then $f$ is an isomorphism of stacks.
\end{lemma}

\subsubsection{The stacks are isomorphic}

To conclude, we prove that (\ref{morphism-main}) above is an isomorphism.  

\begin{theorem}\label{prop:comparison}
	The morphism $f \colon \mathcal{I}^n_{X/C} \to
	\mathcal{I}^n_{\mathfrak{X}/\mathfrak{C}} $ is an isomorphism of
	Deligne-Mumford stacks.
\end{theorem}

\begin{proof}
	First we observe that $f$ is representable. Indeed, this follows from
	\cite[Lem.~6]{AK13}, because $\mathcal{I}^n_{X/C}$ has finite inertia (being
	a separated DM-stack), and because $f$ yields an isomorphism of automorphism
	groups for all geometric points. The second property is due to the fact that
	the formation of the standard models $X[n] \to C[n]$ commutes with base
	change to any algebraically closed overfield of $k$, together with a similar
	argument as in Lemma \ref{lemma-pointsstab}.
	
	Moreover, $f$ is of finite type and \'etale. Since we have already
	established that $f$ is smooth, it suffices to prove that it is unramified.
	This can be checked on geometric points, and is a direct computation. 
	
	Since $f$ is representable, it suffices to prove, for any \'etale atlas $Y$
	of $ \mathcal{I}^n_{\mathfrak{X}/\mathfrak{C}} $, that the pullback $f_Y$ of
	$f$ is an isomorphism of schemes. We claim that $f_Y$ is in fact a surjective
	open immersion. Indeed, this follows from Lemma \ref{lemma-pointsstab}
	together with Lemma \ref{lemma-stackstech}.
\end{proof}


\section{Example}\label{sec:examples}

In this section we want to discuss one example in detail in order to demonstrate
how our machinery works. We start with a simple degeneration $X \to C$ where the
central fibre $X_0=Y_1 \union Y_2$  has two components intersecting along a
smooth irreducible subvariety $D=Y_1 \intsct Y_2$. We want to explain the
geometry of the degenerate Hilbert scheme for $n$ points. For most of this
discussion the dimension  of the fibres will be irrelevant, so we will allow it
to be arbitrary for the time being. In this case the dual graph
$\Gamma=\Gamma(X_0)$ is simply 
\begin{equation}\label{equ:twocompgraph}
	\bullet \xrightarrow{\gamma} \bullet
\end{equation}
which is trivially a bipartite graph.

Recall the expanded degenerations $X[n] \to C[n]$. If $t\colon C \to \AA^1$ is a
local \'etale coordinate, then we obtain a map $(t_1, \ldots, t_{n+1}) \colon
C[n] \to \AA^{n+1}$. Let $I=\{a_1, \ldots , a_{r}\} \subset [n+1]$ and denote by
$X[n]_I$ the locus of $X[n]$ which is the pre-image of the subscheme $C[n]_I$
where $t_{a_i}=0, a_i\in I$. In Proposition \ref{prop:components} we analysed the
components of $X[n]_I$ and  found that they correspond  to the vertices of a
graph $\Gamma_I$ which is derived from  $\Gamma$ by replacing each edge $\gamma$
by new edges labelled $\gamma_{a_1}, \ldots, \gamma_{a_{r}}$, arranged in
increasing order, and inserting white vertices at the ends of $\gamma_{a_1},
\ldots , \gamma_{a_{r-1}} $. Since in our case $\Gamma$ only has one edge
$\gamma$ we can omit this from our notation and simply relabel the edges
$\gamma_{a_{\ell}}$ by $a_\ell$. The graph $\Gamma_I$ thus becomes 
\begin{equation} \label{equ:expandedgraphcontr}
	\bullet \xrightarrow{a_1} \circ \xrightarrow{a_2} \circ
	\cdots \circ  \xrightarrow{a_{r}} \bullet.
\end{equation}  

The extremal case is given by $I=I_{\max}=[n+1]$, in which case we arrive at the
graph $\Gamma_{I_{\max}}$ given by
\begin{equation} \label{equ:expandedgraphtwo}
	\bullet \xrightarrow{1} \circ \xrightarrow{2} \circ
	\cdots \circ \xrightarrow{n+1} \bullet.
\end{equation}  
All other graphs $\Gamma_I$ with $I \subset I_{\operatorname {max}}$ arise from
$\Gamma_{I_{\operatorname {max}}}$ by deleting the arrows in $I_{\operatorname
{max}} \setminus I$. By Proposition \ref{prop:components} we have a
decomposition into irreducible components
\begin{equation*}
	X[n]_I = \Delta^{0}_{I} \union
	\ldots \union \Delta^{a_{\ell}}_{I}  \union \ldots  \union \Delta^{a_{r}}_{I}
\end{equation*}
where $\Delta^{0}_{I}$ and $\Delta_I^{a_{r}}$ belong to the black vertices of
the graph (\ref{equ:expandedgraphcontr})  while the components
$\Delta^{a_{\ell}}_{I}, \ell= 1, \ldots r-1$ correspond to the white vertices.
Note that since there is only one component $D$, we have dropped $D$ from the
notation and have thus set  $\Delta^{D,a_{\ell}}_I= \Delta^{a_{\ell}}_I$. Under
the natural projection $X[n] \to X \times_C C[n]$ the components
$\Delta^{0}_{I}$ and $\Delta^{a_r}_I$ are  mapped birationally onto $Y_1
\times_C C[n]_I$ and $Y_2 \times_C C[n]_I$ respectively. The components
$\Delta^{a_{\ell}}_{I}, \ell= 1, \ldots r-1$ are contracted to $D \times_C
C[n]_I$. The latter are the inserted components which have the structure of a
$\PP^1$-bundle, whose fibres are contracted under the map to $D \times_C
C[n]_I$. 
 
There is another way of labelling the components of $X[n]_I$ which is sometimes
helpful in geometric considerations. If $I=\{a_1, \ldots, a_{r}\}$, then we
decompose $I_{\operatorname {max},0}= \{0\} \union [n+1]$ into $I_{\operatorname
{max},0}= I_0 \union I_1 \union \ldots \union I_{r}$ where $I_0=[0,a_1-1]$,
$I_{\ell}=[a_{\ell}, a_{\ell+1}-1]$ for $1 \leq \ell \leq r-1$ and
$I_r=[a_r,n+1]$. The components $ \Delta^{a_{\ell}}_I $ then correspond to the
first entry in each interval $I_{\ell}$. We can understand the above graph
(\ref{equ:expandedgraphcontr}) as a contraction of the maximal graph
$\Gamma_{I_{\max}}$ given in (\ref{equ:expandedgraphtwo}) by identifying all the
edges labelled in one of the sets $I_{\ell}$ in the partition $I_{\operatorname
{max},0}= I_0 \union I_1 \union \ldots \union I_{r}$. So we can symbolically
think of the left hand bold vertex of (\ref{equ:expandedgraphcontr}) as
\begin{equation*}
	\bullet \stackrel{1}{=} \circ \cdots \circ \stackrel{a_1-1}{=} \circ
	\xrightarrow{a_1} \circ \cdots
\end{equation*}
the middle white vertices as 
\begin{equation*}
	\cdots \xrightarrow{a_{\ell}} \circ \stackrel{a_{\ell}+1}{=} \circ \cdots
	\circ \stackrel{a_{\ell+1}-1}{=} \circ \xrightarrow{a_{\ell}+1} \cdots
\end{equation*}
and finally the right hand bold vertex as 
\begin{equation*}
	\cdots \xrightarrow{a_{r}} \circ \stackrel{a_{r}+1}{=} \circ \cdots \circ
	\stackrel{n+1}{=} \bullet.
\end{equation*}

This picture also helps us understand the smoothing or, in other words, the
inclusion of the closure of the strata when we move from $t_{a_{\ell}}=0$ to
$t_{a_{\ell}} \neq 0$. This corresponds to removing $a_{\ell}$ from the set $I$
or, equivalently, to replacing $I_{\ell-1}$ and $I_{\ell}$ by their union
$I_{\ell-1} \union I_{\ell}$.

Now consider a subscheme $Z$ of length $n$ representing a point in the relative
Hilbert scheme $ \mathbf{H}^n = \Hilb^n(X[n]/C[n]) $. Since $ \mathbf{H}^n$ is
the relative Hilbert scheme, every subscheme $Z$ lies in some fibre $\mathbf{H}^n_q$
for a point $q\in C[n]$. Let $I$ be the the set of indices labelling the
coordinates $t_{a_i}$ which vanish at $q$. In Section \ref{sec:GIT-analysis} we
developed a numerical criterion for stability. First of all recall that
stability and semi-stability coincide. Moreover, all stable cycles have support
in the smooth part $X[n]_I^{\circ}$ of $X[n]_I$, by which we mean that $Z$ does
not intersect the locus where different components of $X[n]_I$ meet. We shall
denote the restriction of the smooth locus $X[n]_I^{\circ}$ to the components
$\Delta^{a_{\ell}}_I$ by $\Delta^{a_{\ell},\circ}_I$. We now claim that the
numerical criterion of Theorem \ref{theorem-stablelocus} is equivalent to
\begin{equation} \label{equ:example1}
	Z \subset X[n]_I  { \mbox { is stable} }
	\Leftrightarrow
	\operatorname{length} (Z \intsct \Delta^{a_{\ell},\circ}_I)
	= \vert I_{\ell} \intsct [n]  \vert \ \forall \, \ell.
\end{equation}
Indeed, in the notation of Section \ref{sec:GIT-analysis} we have  $\mathbf{a} =
(1,a_1, \ldots ,a_{r},n+1)$ and thus   $\mathbf{v}_{\mathbf{a}}=(a_1-1, a_2-a_1,
\ldots, a_{r}-a_{r-1}, n+1-a_{r})$. The stability condition  of Theorem
\ref{theorem-stablelocus} for a cycle $Z$ is
$\mathbf{v}(Z)=\mathbf{v}_{\mathbf{a}}$ where $\mathbf{v}(Z)$ is the numerical
support of $Z$, i..e. the length of the cycle restricted to the smooth part
$\Delta^{a_{\ell},\circ}_I$ of the components $\Delta^{a_{\ell}}_I$ of
$X[n]_I$. The claim now follows since the entries of  $\mathbf{v}_{\mathbf{a}}$
are exactly equal to the cardinality of the sets $I_{\ell} \intsct[n]$. 

Our aim is to understand the geometry of the GIT quotient $I^n_{X/C} =
\mathbf{H}^n_{\mathrm{GIT}}/G[n]$, in particular the geometry of the special
fibre $(I^n_{X/C})_0$. Since the Hilbert schemes of varieties of dimension
greater than $2$ are, in general, neither irreducible nor equi-dimensional, we
will for the following discussion restrict the fibre dimension to $d\leq 2$. We
first observe that the fibre $(I^n_{X/C})_0$ is naturally stratified. As we have
seen, any  length $r$  subset $I=\{a_1, \ldots, a_{r}\} \subset [n+1]$ defines a
subscheme $X[n]_I$ of $X[n]$ and the stable $n$ cycles supported on $X[n]_I$
give rise to a stratum $(I^n_{X/C})_I$ of $(I^n_{X/C})_0$, and it is the
geometry of these strata and the inclusion relations of their closures which we
want to describe here.

We start with the case where $I=\{a_1\}$ consists of one element. In this case
$I$ defines a partition of $I_{\operatorname{max},0}=I_0 \union I_1$ into two
intervals, namely $I_0=[0, a_1-1]$ and $I_1=[a_1, n+1]$. The
graph $\Gamma_I$ then becomes
\begin{equation*} 
	\bullet \stackrel{1} = \circ \cdots \stackrel{a_1-1}{=} \circ
	\stackrel{a_1}{\to} \circ \stackrel{a_{1}+1}{=} \circ \cdots \circ
	\stackrel{n+1}{=} \bullet
\end{equation*}
and we have no inserted components. The general fibre of $X[n]_I$ has two
components, which are isomorphic to $Y_1$ and $Y_2$ respectively. Stability
condition (\ref{equ:example1}) then tells us that we must have $a_1-1$ points on
$Y_1$ and $n+1-a_1$ points on $Y_2$.  In this case the group $G[n]$ acts freely
on the base $C[n]_I$ of the fibration $X[n]_I \to C[n]_I$. Varying $a_1$ from
$1$ to $n+1$ we thus obtain the strata  $\Hilb^{a_1-1}(Y_1^{\circ}) \times
\Hilb^{n+1-a_1}(Y_2^{\circ})$ in the quotient, where $Y_i^{\circ}$ denotes open
set away from the intersection $D=Y_1 \intsct Y_2$.

Next we consider the other extremal case, namely where $I$ is maximal, i.e.
$I=I_{\operatorname{max}} = [n+1]$. In this case $I_{\operatorname{max},0}$ is
partitioned into $n+2$ subsets $[\{0\}, \{1\}, \ldots, \{n+1\}]$ and the
associated graph is as in (\ref{equ:expandedgraphtwo}). Stability condition
(\ref{equ:example1}) then says that $Z$ must have one point on each of the $n$
inserted components, and consequently none on the components $Y_1$ or $Y_2$.
Recall that the fibres of every inserted component $\Delta^{a}_I$, $a=1, \ldots,
n$ are $\PP^1$-bundles over $D$ and that the smooth locus $\Delta^{a,\circ}_I$
is a  $\Gm$ fibration, given by removing the $0$-section and the
$\infty$-section of the $\PP^1$-bundle. Since stable cycles lie in the smooth
part of $X[n]_I$ it follows that $Z=(P_1, \ldots, P_n) \in \Delta^{1,\circ}_I
\times \ldots \times  \Delta^{n,\circ}_I$ with $P_i \in \Delta^{i,\circ}_I$.
Here the torus $G[n]$ acts trivially on $C_I$ and transitively by multiplication
on the product $\Gm^n$ of the fibres of $\Delta^{1,\circ}_I \times \ldots \times
\Delta^{n,\circ}_I$ over a given point of $D$, see Section
\ref{sec:localcoordinates} for details. Hence the stable cycles in $X[n]_I$
map to an $n$-dimensional stratum $D^n$ in $(I^n_{X/C})_0$. 

Now let us consider the general case $I=\{a_1, \ldots, a_{r}\}$. In this case we
have $r-1$ inserted components $\Delta^{a_{\ell}}_I$, $\ell= 1, \ldots, r-1$. By
the calculations of \ref{sec:localcoordinates} the group $G[n]$ has a
subgroup $G[k]$ which acts trivially on $C[n]_I$ and  transitively by
multiplication on the fibres of  $\Delta^{\circ}_{I_1} \times \ldots \times
\Delta^{\circ}_{I_k}$, whose product, over each point in $D$, is isomorphic to
$\Gm^k$.  In this case we obtain quotients of products of the form
$\Hilb^{a_1-1}(Y_1^{\circ}) \times \Hilb^{a_2-a_1}(\Delta_I^{a_1,\circ}) \times
\ldots \times \Hilb^{a_{r}-a_{r-1}}(\Delta_{I}^{a_{r-1},\circ}) \times
\Hilb^{n+1-a_{r}}(Y_2^{\circ})$ by the group $G[k]$.

The above description provides a natural stratification of $(I^n_{X/C})_0$ into
locally closed subsets $(I^n_{X/C})_I$ indexed by the subsets $I \subset
I_{\operatorname {max},0}$. Moreover, we can also describe how these strata are
related with respect to inclusion, namely
\begin{equation*}
	(I^n_{X/C})_J \subset \overline{(I^n_{X/C})}_I \Leftrightarrow I \subset J.
\end{equation*}
It is natural to encode this information about the strata of $(I^n_{X/C})_0$,
together the incidence relation of their closures, in a dual complex. In our
example the situation is very simple: the $k$-simplices are in $1:1$
correspondence to the subsets $I \subset I_{\operatorname{max}}$ of length $k+1$
and the simplex corresponding to $I$ is contained in the simplex corresponding
to $J$ if and only if  $I \subset J$. Hence the resulting dual complex is the
standard $n$-simplex. The maximal $n$-dimensional cell corresponds to the
smallest stratum, which is isomorphic to $D^n$, and the $0$-vertices correspond
to the maximal-dimensional strata $\Hilb^{a-1}(Y_1^{\circ}) \times
\Hilb^{n+1-a}(Y_2^{\circ})$, $a=1, \ldots, n+1$. 

It is interesting to ask which dual complexes one obtains for more general
degenerations. Given a degeneration graph $\Gamma$ for a degeneration of curves
or surfaces, one can indeed define a suitable $\Delta$-complex, see
\cite{RS-1971}, and describe its combinatorial properties. We are planning to
return to this in a future paper. Similarly, one can ask the same question for
higher $d$-dimensional degenerations. As long as the degree $n \leq 3$, the
Hilbert scheme is irreducible and smooth of dimension $dn$ and one can hope for
an interesting combinatorial object. For arbitrary dimension $d$ and degree $n$
the situation will become much more complicated as the Hilbert schemes, even of
smooth varieties, are in general neither irreducible nor even equi-dimensional. 

Finally, we want to say a few words about the singularities of the total space
$I^n_{X/C}$ and, for the case of simplicity, we will restrict ourselves to
degree $2$ Hilbert schemes, and we will thus allow the dimension $d$ of the
fibres to be arbitrary again. Since $X[2]$ is smooth and all semi-stable points
are stable, the quotient is also smooth at orbits where $G[2]$ acts freely. This
is an easy consequence of Luna's slice theorem, see \cite[Proposition
5.8]{drezet-2004}. In order to understand the set of stable points with
nontrivial stabilizer we look at the various strata $X[2]_I$. Clearly $G[2]$
acts freely at points of $C[2]$, and hence also at points of $X[2]$, where all
$t_{a_i}\neq 0$. The same is true if exactly one $t_{a_i}=0$, i.e. if $|I|=1$.  If
$|I|=3$, then our above discussion shows that all stable points are of the form
$Z=(P_1, P_2) \in \Delta^{1,\circ}_I \times \Delta^{2,\circ}_I$. Moreover, by
Section \ref{sec:localcoordinates} we know that $G[2]$ acts transitively and
freely  by multiplication on each fibre $\Gm^2$ of  $\Delta^{1,\circ}_I \times
\Delta^{2,\circ}_I$ over a given point of $D$.

It thus remains to consider the case where $|I|=2$. We first consider
$I=\{2,3\}$. Then we have the partition $[\{0,1\}, \{2\}, \{3\}]$ and one
inserted component $\Delta^2_I$. By the stability condition (\ref{equ:example1})
every stable cycle $Z$ must contain a point in $\Delta^2_I$. The stabilizer of
points in $C[2]$ with $t_2=t_3=0$ and $t_1 \neq 0$ is the rank $1$ subtorus
$G[1] \subset G[2]$ given by $\sigma_1=1$. However, by Proposition
\ref{prop:localeq} this stabilizer acts on the fibres of $\Delta^2_I$ by
$(u_2:v_2) \mapsto (\sigma_2u_2:v_2)$. Hence $G[2]$ acts freely on the stable
cycles supported on $X[n]_{I}$. A similar argument applies to $I=\{1,2\}$ and it
thus remains to consider $I=\{1,3\}$. In this case we have one inserted
component $\Delta^1_I$ and by the stability condition every stable $2$-cycle $Z$
is supported on it. To study the non-free locus and the action of the stabilizer
we work on the chart $W_2$ from Remark \ref{rem:tau-coord}. where we have the
coordinates $(t_1,t_2,t_3,x_2, \ldots, x_d,u_1/v_1,u_2/v_2)$ and the relation
$t_2= (u_2/v_2)\cdot (v_1/u_1)$. Since stable cycles are supported on the smooth
locus $\Delta^{1,\circ}_I$ we have $u_1/v_1 \neq 0$ and we can thus eliminate
$u_2/v_2$ as a coordinate working with  $(t_1,t_2,t_3, x_2, \ldots ,x_d,
u_1/v_1)$. Here $x_2, \ldots, x_d$ are coordinates on $D$ and the group $G[2]$
acts trivially on these coordinates. For simplicity we  write $U=u_1/v_1$. Thus
the action on our coordinates is given by
\begin{equation*}
	(t_1,t_2,t_3,x_2, \ldots, x_d,U)
	\mapsto
	(\sigma_1t_1, \sigma_2t_2, (\sigma_1\sigma_2)^{-1}t_3, x_2, \ldots ,x_d,  \sigma_1U).
\end{equation*}
Since $t_2 \neq 0$, any element in a nontrivial stabilizer must necessarily
have $\sigma_2=1$. In particular, any nontrivial stabilizer group must lie in
the rank $1$ torus $G_1[2]=\langle \sigma_1 \rangle \subset G[2]$. This group
acts freely on $\Delta^{1,\circ}_I$. Hence the only points in the relative
degree $2$ Hilbert schemes which can possibly have nontrivial stabilizers must
be pairs of points $\{(x_2, \ldots, x_d,U), (x_2, \ldots, x_d,V)\}$ with
$\sigma_1U=V$ and $\sigma_1V=U$. This implies $\sigma_1 = \pm 1$ and $U+V=0$. In
particular, the corresponding point in the degree $2$ Hilbert scheme is
represented by a reduced $2$-cycle and thus, when analysing the action of the
stabilizer group, we can work with the relative second symmetric product rather
than the Hilbert scheme. In order to describe this in coordinates we introduce a
second set of fibre coordinates $(y_2, \ldots ,y_d,V)$. Forming the relative
second symmetric product  means factorizing by the involution which interchanges
$x_i$ and $y_i$ as well as $U$ and $V$. The invariants under this involution are
generated by the linear invariant forms $A_i=x_i+y_i, i=2, \ldots,d$ and
$B=U+V$ as well as the quadratic forms $C_{ij}=(x_i-y_i)(x_j-y_j), 2 \leq i,j
\leq d$, $D_j=(x_j-y_j)(U-V), j= 2, \ldots, d$ and $E=(U-V)^2$. The relations
among these are generated by $C_{ij}E=D_iD_j$. The fixed points lie on  $U+V=0$,
so we can assume that $E\neq 0$ near the fixed points. Thus we can eliminate
$C_{ij}$ and work with the coordinates given by $A_i,B,D_j,E$ where  $i,j=2,
\ldots,d$. On these coordinates the torus $G_1[2]$ acts as 
\begin{equation*}
	(t_1,t_2,t_3,A_i,B,D_j,E)
	\mapsto
	(\sigma_1 t_1,t_2,\sigma_1^{-1} t_3, A_i,\sigma_1B, \sigma_1 D_j, \sigma_1^2E ).
\end{equation*}
From this we see immediately that the differential of involution given by $\{
\pm 1 \} \subset G[2]$ is a diagonal matrix with $3+d-1=d+2$ entries $-1$ and
$d+1$ entries $1$. It then follows from Luna's slice theorem \cite[Theorem
5.4]{drezet-2004} that the quotient $I^n_{X/C}$ has a transversal singularity
along $D$ of type $\frac12(1,\ldots, 1)$ where we have $d+2$ entries $1$. This
singularity is the cone over the Veronese embedding of ${\mathbb P}^{d+1}$
embedded by  the linear system $|{\mathcal O}_{{\mathbb P}^{d+1}}(2)|$. We also
note that in the case $d=1$ we mistakenly labelled this an $A_1$-singularity in
\cite[Example 6.2]{GHH-2015}.


\bibliography{sympdeg}{}
\bibliographystyle{alpha}

\end{document}

%% file: expanded.pspdftex
\begin{picture}(0,0)%
\includegraphics{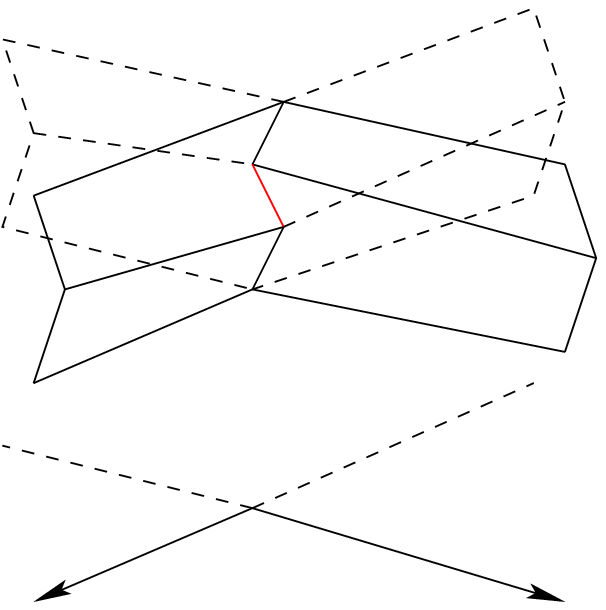}%
\end{picture}%
\setlength{\unitlength}{3947sp}%
\begingroup\makeatletter\ifx\SetFigFont\undefined%
\gdef\SetFigFont#1#2#3#4#5{%
  \reset@font\fontsize{#1}{#2pt}%
  \fontfamily{#3}\fontseries{#4}\fontshape{#5}%
  \selectfont}%
\fi\endgroup%
\begin{picture}(2874,3156)(1039,-2305)
\put(1201,-2236){\makebox(0,0)[lb]{\smash{{\SetFigFont{12}{14.4}{\rmdefault}{\mddefault}{\updefault}{\color[rgb]{0,0,0}$t_1$}%
}}}}
\put(3526,-2236){\makebox(0,0)[lb]{\smash{{\SetFigFont{12}{14.4}{\rmdefault}{\mddefault}{\updefault}{\color[rgb]{0,0,0}$t_2$}%
}}}}
\end{picture}%

%% file: nonbipartite.pspdftex
\begin{picture}(0,0)%
\includegraphics{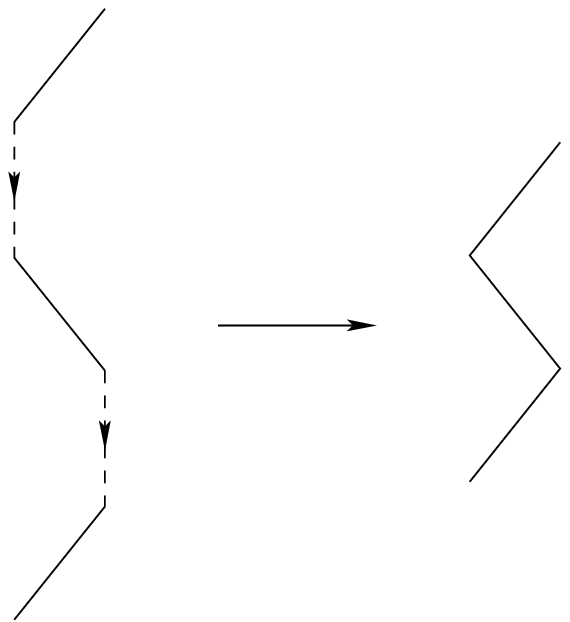}%
\end{picture}%
\setlength{\unitlength}{3947sp}%
\begingroup\makeatletter\ifx\SetFigFont\undefined%
\gdef\SetFigFont#1#2#3#4#5{%
  \reset@font\fontsize{#1}{#2pt}%
  \fontfamily{#3}\fontseries{#4}\fontshape{#5}%
  \selectfont}%
\fi\endgroup%
\begin{picture}(2864,2956)(436,-2405)
\put(885,-873){\makebox(0,0)[lb]{\smash{{\SetFigFont{9}{10.8}{\rmdefault}{\mddefault}{\updefault}{\color[rgb]{0,0,0}$Y_2$}%
}}}}
\put(1863,-873){\makebox(0,0)[lb]{\smash{{\SetFigFont{9}{10.8}{\rmdefault}{\mddefault}{\updefault}{\color[rgb]{0,0,0}$\pi$}%
}}}}
\put(3167,-1633){\makebox(0,0)[lb]{\smash{{\SetFigFont{9}{10.8}{\rmdefault}{\mddefault}{\updefault}{\color[rgb]{0,0,0}$Y_3$}%
}}}}
\put(994,-2176){\makebox(0,0)[lb]{\smash{{\SetFigFont{9}{10.8}{\rmdefault}{\mddefault}{\updefault}{\color[rgb]{0,0,0}$Y_3$}%
}}}}
\put(777,322){\makebox(0,0)[lb]{\smash{{\SetFigFont{9}{10.8}{\rmdefault}{\mddefault}{\updefault}{\color[rgb]{0,0,0}$Y_1$}%
}}}}
\put(451,-330){\makebox(0,0)[lb]{\smash{{\SetFigFont{9}{10.8}{\rmdefault}{\mddefault}{\updefault}{\color[rgb]{0,0,0}$\Delta_1$}%
}}}}
\put(885,-1524){\makebox(0,0)[lb]{\smash{{\SetFigFont{9}{10.8}{\rmdefault}{\mddefault}{\updefault}{\color[rgb]{0,0,0}$\Delta_2$}%
}}}}
\put(2949,-330){\makebox(0,0)[lb]{\smash{{\SetFigFont{9}{10.8}{\rmdefault}{\mddefault}{\updefault}{\color[rgb]{0,0,0}$Y_1$}%
}}}}
\put(3058,-873){\makebox(0,0)[lb]{\smash{{\SetFigFont{9}{10.8}{\rmdefault}{\mddefault}{\updefault}{\color[rgb]{0,0,0}$Y_2$}%
}}}}
\put(2732,-764){\makebox(0,0)[lb]{\smash{{\SetFigFont{9}{10.8}{\rmdefault}{\mddefault}{\updefault}{\color[rgb]{0,0,0}$p_1$}%
}}}}
\put(3275,-1307){\makebox(0,0)[lb]{\smash{{\SetFigFont{9}{10.8}{\rmdefault}{\mddefault}{\updefault}{\color[rgb]{0,0,0}$p_2$}%
}}}}
\end{picture}%